\documentclass[psamsfonts]{amsart}
\usepackage[utf8]{inputenc}

\usepackage{amssymb,amsfonts}
\usepackage{enumerate}
\usepackage{mathrsfs}
\usepackage{url}
\usepackage{mathtools}
\usepackage[font=small,labelfont=bf]{caption}
\usepackage[all,arc]{xy}
\usepackage{xcolor}
\usepackage{scalerel}
\usepackage{stackengine,wasysym}
\usepackage{xcolor,url}
\usepackage{hyperref}
\usepackage{todonotes}
\hypersetup{
    colorlinks,
    citecolor=blue,
    filecolor=blue,
    linkcolor=black,
    urlcolor=blue
}

\newtheorem{thm}{Theorem}[section]
\newtheorem{cor}[thm]{Corollary}
\newtheorem{prop}[thm]{Proposition}

\newtheorem{lem}[thm]{Lemma}

\newtheorem{quest}[thm]{Question}

\theoremstyle{definition}
\newtheorem{defn}[thm]{Definition}

\newtheorem{fact}[thm]{Fact}

\newtheorem*{theorem*}{Theorem}

\theoremstyle{remark}
\newtheorem{rem}[thm]{Remark}

\makeatletter
\let\c@equation\c@thm
\makeatother
\numberwithin{equation}{section}

\def\Ind{\setbox0=\hbox{$x$}\kern\wd0\hbox to 0pt{\hss$\mid$\hss} \lower.9\ht0\hbox to 0pt{\hss$\smile$\hss}\kern\wd0} 

\def\Notind{\setbox0=\hbox{$x$}\kern\wd0\hbox to 0pt{\mathchardef \nn=12854\hss$\nn$\kern1.4\wd0\hss}\hbox to 0pt{\hss$\mid$\hss}\lower.9\ht0 \hbox to 0pt{\hss$\smile$\hss}\kern\wd0}

\title[pseudo-finite permutation groups]{Primitive pseudo-finite permutation groups of finite $SU$-rank}

\thanks{$^\dagger$ Karhum\"{a}ki is supported by the Finnish Science Academy grant no: 338334.  }
\thanks{$^\ddagger$ Ramsey is supported by NSF grant DMS-2246992.  }

\author[U. Karhum\"aki]{Ulla Karhum\"aki$^\dagger$}
\address{Department of Mathematics \\
University of Helsinki \\ Helsinki }
\email{ulla.karhumaki@helsinki.fi}

\author[N. Ramsey]{Nicholas Ramsey$^{\ddagger}$}
\address{Department of Mathematics \\
University of Notre Dame\\
 USA}
\email{sramsey5@nd.edu}

\date{\today}

\begin{document}

\maketitle

\begin{abstract}
We study definably primitive pseudo-finite permutation groups of finite $SU$-rank.  We show that if $(G,X)$ is such a permutation group, then the rank of $G$ can be bounded in terms of the rank of $X$, providing an analogue of a theorem of Borovik and Cherlin in the setting of definably primitive permutation groups of finite Morley rank. 
\end{abstract}

\setcounter{tocdepth}{1}
\tableofcontents

\section{Introduction}
A \emph{permutation group} $(G,X)$ is a group $G$ together with a faithful action of $G$ on the set $X$. A permutation group is called \emph{primitive} if the action of $G$ preserves no nontrivial equivalence relation on $X$.  The primitive permutation groups are the basic building blocks for all permutation groups and, for this reason, their classification, in various categories, is often of considerable interest. Within a model-theoretic context, it is natural to consider permutation groups $(G,X)$ where the group $G$, the set $X$, and the action $G \curvearrowright X$ are all definable in some tame structure. In this setting, a permutation group is called \emph{definably primitive} if the action preserves no nontrivial \emph{definable} equivalence relation. Definably primitive permutation groups of finite Morley rank were studied in detail by Macpherson and Pillay \cite{Macpherson-Pillay1995} and later by Borovik and Cherlin \cite{BC08}, who proved strong structure theorems for them, and this work inspired several generalisations to different model-theoretic contexts. Motivated by this work, this paper studies definably primitive pseudo-finite permutation groups of finite $SU$-rank, showing that they are quite constrained. These groups are a core example of groups in simple theories and also are connected to applications to, e.g., binding groups in the theory ACFA of existentially closed difference fields. 

In particular, we prove that, in a definably primitive pseudo-finite permutation group $(G,X)$ of finite $SU$-rank, one can bound the $SU$-rank of $G$ in terms of the $SU$-rank of $X$.  This is an analogue of a theorem of Borovik and Cherlin, who proved there is a function $f:\mathbb{N} \to \mathbb{N}$ such that, in any definably primitive permutation group $(G,X)$ of finite Morley rank, $\mathrm{RM}(G) \leq f(\mathrm{RM}(X))$. The question of whether such a bound could be given for pseudo-finite permutation groups of finite $SU$-rank was asked by Macpherson \cite[Problem 5.0.15]{Macpherson2018} and, in a somewhat more limited form, by Elwes, Jaligot, Macpherson and Ryten \cite{Elwesetal}. The challenge of proving such theorems in the finite Morley rank setting is that, although the rank is extremely well-behaved, one has incomplete knowledge of the simple groups that can appear, since the Cherlin-Zilber conjecture, which predicts that simple groups of finite Morley rank are algebraic groups over algebraically closed fields, is still open. Working instead with pseudo-finite permutation groups of finite $SU$-rank makes things simultaneously easier and harder: the notion of rank is less well-behaved (there is, for example, no sensible notion of multiplicity) but, thanks to the classification of finite simple groups (CFSG), we have almost perfect knowledge of what simple groups in this setting can look like. Our main theorem is the following.

\begin{theorem*}[Theorem \ref{th: main}]
If $(G,X)$ is a pseudo-finite definably primitive permutation group of finite $SU$-rank, then $SU(G)$ can be bounded as a function of $SU(X)$. 
 
More precisely, if $(G,X)$ is a pseudo-finite definably primitive permutation group of finite $SU$-rank then, setting $r=SU(X)$, then the following holds. \begin{enumerate}
\item If ${\rm Rad}(G)\neq 1$ then $SU(G) \leqslant r + (r^2+1)r$.
\item If ${\rm Rad}(G)= 1$ then one of the following holds. \begin{enumerate}
\item $(G,X)$ is an almost simple group and $SU(G)  \leq 8r^{2} + 2r$.
\item $(G,X)$ is of simple diagonal action type and $SU(G) \leqslant 2r$.
\item $(G,X)$ is of product action type and $SU(G)  \leq 8r^{2}+2r$.
\end{enumerate} 
\end{enumerate}
\end{theorem*}

In a rough outline, our strategy has two parts.  In the first, we reduce to the analysis of permutation groups which are outright primitive, rather than merely definably so. This is where neostability theory, in the form of the theory of groups in simple theories, enters the picture. We prove that if $(G,X)$ is a definably primitive pseudo-finite permutation group of finite $SU$-rank, then $(G,X)$ is primitive if and only if the point stabiliser $G_{x}$ is infinite. Variants of one direction, showing that definable primitivity implies primitivity for permutation groups with large point stabilisers, have appeared in the literature, starting with \cite{Macpherson-Pillay1995} in the finite Morley rank context, Elwes and Ryten for measurable theories \cite[Proposition 6.1]{elwes2008measurable}, and Elwes, Jaligot, Macpherson, and Ryten in pseudo-finite finite $SU$-rank theories with elimination of the quantifier $\exists^{\infty}$ \cite[Lemma 2.6]{Elwesetal}. We build on this earlier work, making use of the theory of groups in supersimple theories, especially the `almost normaliser' and an indecomposability result. For the other direction, we apply a recent result of Smith \cite{Smith2015} that if $(G,X)$ is an infinite permutation group with $G_{x}$ finite, then $G$ must be finitely generated; we are able to rule out this possibility in our setting with a direct argument. So we prove the following.

\begin{theorem*}[Theorem \ref{th:primitive}]\label{th:primitive}Suppose $(G,X)$ is a supersimple pseudo-finite definably primitive permutation group of finite $SU$-rank. Then $(G,X)$ is primitive if and only if $G_x$ is infinite.
\end{theorem*}

The second part consists of a case analysis, leveraging supersimplicity and classification results for primitive pseudo-finite permutation groups to divide into a small number of situations of a reasonably transparent form. The classical O'Nan-Scott theorem gives a rough classification of finite primitive permutation groups $(G,X)$ into six categories, based both on the form of the socle of $G$ and on the shape of the action.  More recently, a remarkable classification theorem of Liebeck, Macpherson, and Tent \cite{LMT} characterises the families of finite permutation groups, in each of the O'Nan-Scott categories, whose ultraproducts are primitive. Supersimplicity winnows down the Liebeck-Macpherson-Tent catalogue even further. For example, in a supersimple structure, one cannot interpret a group which is an ultraproduct of finite classical groups of unbounded Lie ranks. A case-by-case analysis provided by the Liebeck-Macpherson-Tent result in our set-up allows us to get the bounds in Theorem~\ref{th: main}.

This paper is organised as follows. In Section~\ref{sec:prelim} we give all necessary background results. In particular, we give a version of the Liebeck-Macpherson-Tent classification which is suitable for our purposes (Theorem~\ref{thm:LMT}). Theorem~\ref{th:primitive} is proven in Section~\ref{sec:prim}, where we also give an example of a definably primitive non-primitive pseudo-finite permutation group of finite $SU$-rank, which has nontrivial (finite) point stabilisers. Then, in Section~\ref{sec:bounds}, we prove Theorem~\ref{th: main} via a case-by-case analysis provided by Theorems~\ref{thm:LMT} and~\ref{th:primitive}. Finally, in Section~\ref{sec:trank1} we note that we recover the classification of definably primitive pseudo-finite permutation groups of finite $SU$-rank acting on rank 1 sets (Theorem~\ref{th:rank1}), which is known by the main results in \cite{Elwesetal, Zou2020}.

\section{Preliminaries}\label{sec:prelim}
Throughout the paper, unless specified otherwise, definable (resp. interpretable) means with parameters. Also, when we say `a simple group' we always mean in terms of group theory (no proper normal nontrivial subgroups).

\subsection{Supersimple and finite-dimensional groups}\label{sec:findim} A group $G$ whose theory ${\rm Th}(G)$ is \emph{supersimple} is a group whose definable sets are equipped with a notion of dimension, called the \emph{$SU$-rank}, taking ordinal values. If the $SU$-ranks of the definable sets of $G$ take values in natural numbers then we say that ${\rm Th}(G)$ is of \emph{finite $SU$-rank}. We say that $G$ is of finite $SU$-rank if its theory ${\rm Th}(G)$ is. A general reference for such groups is \cite{Wagner2000}. In this paper, we mostly consider permutation groups of finite $SU$-rank but, from time to time, we also consider the wider class of \emph{finite-dimensional permutation groups with fine and additive dimension} (\cite{Wagner2020}):

A structure $M$ is called \emph{finite-dimensional} if there is a dimension function ${\rm dim}$ from the collection of all interpretable sets in models of ${\rm Th}(M)$ to $\mathbb{N} \cup \{-\infty\}$ such that, for any formula $\phi(x,y)$ and interpretable sets $X$ and $Y$, the following hold: \begin{enumerate}
\item Invariance: If $a\equiv a'$ then ${\rm dim}(\phi(x,a))={\rm dim}(\phi(x,a'))$.
\item Algebraicity: If $X \neq \emptyset$ is finite then ${\rm dim}(X)=0$, and ${\rm dim}(\emptyset)=-\infty$.
\item Union: ${\rm dim}(X \cup Y)={\rm max}\{{\rm dim}(X), {\rm dim}(Y)\}$.
\item Fibration: If $f: X \rightarrow Y$ is an interpretable map such that ${\rm dim}(f^{-1}(y)) \geqslant d$ for all $y\in Y$ then ${\rm dim}(X)\geqslant {\rm dim}(Y)+d$.
\end{enumerate} The dimension of a tuple $a$ of elements over a set $B$ is defined as
$${\rm dim}(a/B) := {\rm inf}\{ {\rm dim}(\phi(x)) : \phi \in {\rm tp}(a/B) \}.$$ We say that the dimension is \begin{itemize}
\item \emph{additive} if ${\rm dim}(a,b/C) = {\rm dim}(a/b,C) + {\rm dim}(b/C)$ holds for any
tuples $a$ and $b$ and for any set $C$; and
\item \emph{fine} if ${\rm dim}(X)=0$ implies that $X$ is finite. 
\end{itemize}

\begin{fact}[Lascar equality, see {\cite[Fact 2.7]{Karhumaki-Wagner2024}}]\label{fact:Lascar-eq}Let $G$ be a finite-dimensional group with additive and fine dimension and $H \leqslant G$ be a definable subgroup. Then ${\rm dim}(G)={\rm dim}(H)+{\rm dim}(G/H)$.\end{fact}

\subsubsection{Commensurability and almost normalisers} Let $G$ be a group, and $H, K \leqslant G $ be subgroups. The group $H$ is said to be \emph{almost contained} in $K$, written $H \lesssim K$, if $H \cap K$ has finite index in $H$. The subgroups $H,K \leqslant G$ are \emph{commensurable} if both $H \lesssim K$ and $K \lesssim H$. A family $\mathcal{H}$ of subgroups of $G$ is \emph{uniformly commensurable} if there is $n\in \mathbb{N}$ so that $|H_1 : H_1\cap H_2| < n$ for all $H_1,H_2 \in \mathcal{H}$. Likewise, $K \leqslant G$ is uniformly commensurable to $\mathcal{H}$ if and only if $\mathcal{H}$ is uniformly commensurable and $K$ is commensurable to some (equiv.\ any) group in $\mathcal{H}$. The following result is due to Schlichting but the formulation we give here can be found in \cite[Theorem 4.2.4]{Wagner2000}.

\begin{thm}[Schlichting's Theorem]\label{th:commensurable}Let $G$ be a group and $\mathcal{H}$ be a uniformly commensurable family of subgroups of $G$. Then there is a subgroup $N$ of $G$ which is uniformly commensurable with all members of $\mathcal{H}$ and is invariant under all automorphisms of $G$ which fix $\mathcal{H}$ setwise. In fact, $N$ is a finite extension of a finite intersection of elements of $\mathcal{H}$. In particular, if $\mathcal{H}$ consists of definable subgroups, then $N$ is definable. \end{thm}

\begin{defn}Let $G$ be a group, and $H, K \leqslant G $. Then the subgroup $\widetilde{N}_K(H)=\{k\in K : H \text{ is commensurable with }H^k\}$ is the \emph{almost normaliser} of $H$ in $K$. 

The commensurability is \emph{uniform} in $\widetilde{N}_K(H)$ if there is some $m\in \mathbb{N}$ so that if $H \text{ is commensurable with } H^k$ then $|H : H\cap H^k| < m$.
\end{defn}

It is easy to see (\cite[Lemma 2.8]{Karhumaki-Wagner2024}) that any finite-dimensional group with additive and fine dimension (so, in particular, any group which is of finite $SU$-rank) satisfies the following chain condition:

\begin{enumerate}[icc$^0$:]
\item Given a family $\mathcal{H}$ of uniformly definable subgroups of $G$, there is $m < \omega$ so
that there is no sequence $\{H_i : i \leqslant m\} \subset \mathcal{H}$ with $|\bigcap_{i < j}H_i : \bigcap_{i \leqslant j}H_i |\geqslant m$ for all $j \leqslant m$.
\end{enumerate} 

If $G$ is a group satisfying the icc$^0$-condition and $H,K \leqslant G$ are definable, then the commensurability in $\widetilde{N}_K(H)$ is uniform (if necessary, see \cite[Corollary 2.9]{Karhumaki-Wagner2024}); thus $\widetilde{N}_K(H)$ is a definable subgroup of $G$. In particular, such $G$ satisfies the icc$^0$-condition for centralisers and is therefore an \emph{$\widetilde{\mathfrak{M}}_c$-group}:
\begin{defn}A group $G$ satisfies the \emph{$\widetilde{\mathfrak{M}}_c$-condition} if there is $m< \omega$ such that there are no $(g_i: i \leqslant m)$ in $G$ so that $|C_G(g_j: j < i): C_G(g_j: j \leqslant i)|\geqslant m$ for all $i \leqslant m$. If $G$ satisfies the $\widetilde{\mathfrak{M}}_c$-condition then we say that $G$ is an \emph{$\widetilde{\mathfrak{M}}_c$-group}.\end{defn}

\subsection{Simple and semi-simple pseudo-finite groups}\label{sec:semi-simple}

 We denote by $\mathcal{L}_{gp}$ the language of groups.

\begin{defn}A \emph{pseudo-finite} group is an infinite group which satisfies every first-order sentence of $\mathcal{L}_{gp}$ that is true of all finite groups. More generally, an infinite $\mathcal{L}$-structure $M$ is pseudo-finite if whenever $M \models \phi$ for an $\mathcal{L}$-sentence $\phi$, there is some finite $\mathcal{L}$-structure $A$ with $A \models \phi$.  \end{defn}

Note that by \L o\'s' Theorem, an infinite group (resp.\ $\mathcal{L}$-structure) is pseudo-finite if and only if it is elementarily equivalent to a nonprincipal ultraproduct of finite groups ($\mathcal{L}$-structures) of increasing orders.

\begin{rem}Some authors (although virtually none of the sources referenced in this paper) allow pseudo-finite groups to be finite. We require that a pseudo-finite group is infinite. When we write `a (pseudo-)finite group' we mean a group that is either finite or pseudo-finite.\end{rem}

Let $\mathcal{L}_{\mathrm{perm}}$ be the language containing two sorts $G$ and $X$ together with the language of groups on the sort $G$ and a binary function $\cdot : G \times X \to X$. Permutation groups $(G,X)$ may be viewed naturally as $\mathcal{L}_{\mathrm{perm}}$-structures. In what follows, we will assume that all of our structures are definable in some $\mathcal{L}$-structure, where $\mathcal{L}$ is some fixed language extending $\mathcal{L}_{\mathrm{perm}}$. If a permutation group $(G,X)$ is definable in any structure, then the ambient structure has a definitional expansion to a language containing $\mathcal{L}_{\mathrm{perm}}$, so this places no restriction on the scope of our results. Unless otherwise specified, \emph{definable} means definable (possibly with parameters) in the language $\mathcal{L}$.


Using the classification of finite simple groups, the simple pseudo-finite groups have been classified. The following structure theorem mostly follows from the work in \cite{Wilson1995} (Wilson \cite{Wilson1995} proved that a simple pseudo-finite group is elementarily equivalent to a (twisted) Chevalley group $X(F)$ over a pseudo-finite field $F$ and results of Point \cite{Point1999} allow one to conclude that a simple group is pseudo-finite if and only if it is elementarily equivalent to $X(F)$. Ryten \cite[Chapter 5]{Ryten2007} has generalised this by showing that `elementarily equivalent' can be strengthened to `isomorphic'. In particular, he proved the moreover part of the following result.).

\begin{thm}[Wilson \cite{Wilson1995}, Ryten \cite{Ryten2007}]\label{th:wilson} A simple group is
pseudo-finite if and only if it is isomorphic to a (twisted)
Chevalley group over a pseudo-finite field. 

Moreover, the field (in the language of rings) and the group (in the language of groups) are bi-interpretable, and this bi-interpretation is uniform in the Lie type of the group. 
\end{thm}

\begin{rem}\label{rem:alternating}In order to prove the classification above, Wilson showed that a simple pseudo-finite group is elementarily equivalent to a nonprincipal ultraproduct of finite simple groups, and then applied the classification of finite simple groups. In such treatment, the finitely many sporadic groups can be ignored as pseudo-finite groups are infinite. An important fact is that a nonprincipal ultraproduct of alternating groups, say $H$, which is of course definably simple (i.e. has no proper nontrivial definable normal subgroups), is not outright simple since finite alternating groups contain 3-cycles, and
elements of increasingly large support, when written as products of 3-cycles, require
increasingly many 3-cycles. Since in the finite $SU$-rank context non-abelian definably simple groups are actually simple (see the proof of Lemma~\ref{lem:socle} below) we know that a group elementarily equivalent to $H$ cannot be of finite $SU$-rank. This plays a key role in our paper.
\end{rem}

If the largest soluble normal subgroup of a group $G$ exists, then it is called the \emph{soluble radical of }$G$, which is denoted by ${\rm Rad}(G)$. We say that $G$ is \emph{semi-simple} if it has no nontrivial abelian normal subgroup. 
Note that if $A$ is a nontrivial normal abelian subgroup of $G$ and $1\not=a\in A$, then, $Z(C_G(a^G))$ is a definable nontrivial normal abelian subgroup. Thus, semi-simplicity is the same as having no definable nontrivial normal abelian subgroup. By Wilson \cite{Wilson2009}, within the class of finite groups, the soluble radical is \emph{uniformly definable}; this means that there is a parameter free $\mathcal{L}_{gp}$-formula $\phi_R(\bar{x})$ so that, given a pseudo-finite group $G \equiv \prod_{i\in I}G_i/\mathcal{U}$, there is a definable normal subgroup $\phi_R(G) \equiv \prod_{i\in I}\phi_R(G_i)/\mathcal{U}=\prod_{i\in I}{\rm Rad }(G_i)/\mathcal{U}$. Clearly $G$ is semi-simple if and only if $\phi_R(G)=1$. Note however that Wilsons' result depends on the classification of finite simple groups. The following result is independent from CFSG and replaces the use of $\phi_R(\bar{x})$ in this paper:

\begin{thm}[{\cite[Theorem 2.12]{Karhumaki-Wagner2024}}]\label{thm:K-W}If $G$ is finite-dimensional with fine and additive dimension then the soluble radical ${\rm Rad}(G)$ exists and is definable.
\end{thm}

Let $G$ be a semi-simple pseudo-finite group of finite $SU$-rank. Below, using standard methods, we observe that a $G$-minimal definable normal subgroup $N$ of $G$ exists (i.e. an infinite definable normal subgroup $N$ of $G$ which is minimal with respect to these properties). We may then coherently define the \emph{$d$-socle} ${\rm Soc}_d(G)$ of $G$ as the subgroup generated by all definably minimal (i.e. finite or $G$-minimal) normal subgroups of $G$.

Recall that the \emph{socle} ${\rm Soc}(H)$ of a \emph{finite} group $H$ is the subgroup generated by all minimal normal nontrivial subgroups.

\begin{lem}\label{lem:socle}Suppose $G\equiv \prod_{i\in I}G_i/\mathcal{U}$ is a pseudo-finite group of finite $SU$-rank. If ${\rm Rad}(G)=1$ then ${\rm Soc}_d(G) \neq 1$ exists and is elementarily equivalent to $\prod_{i\in I}{\rm Soc}(G_i)/\mathcal{U}$. Further, ${\rm Soc}_d(G)$ is a finite direct product $$M_1 \times \cdots \times M_j$$ of minimal definable normal subgroups of $G$, and, for each $r\in \{1,\ldots, j\}$, we have $$
        M_{r} = S_{r,1} \times \cdots  \times S_{r,\ell_{r}}
        $$ where $S_{r,k}$ are all isomorphic non-abelian simple (pseudo-)finite groups as $r$ varies and $k\in\{1, \ldots , \ell_r\}$, and, at least one of the groups $M_{r}$ is infinite. 
\end{lem}
\begin{proof}Below we argue modulo $\mathcal{U}$ even if implicitly.

The socle ${\rm Soc}(G_i)$ of the finite group $G_i$ is a direct product of $j_i$ many characteristically simple groups $M_{r,i}$ for $r = 1, \ldots, j_{i}$. Since ${\rm Rad}(G)=1$ (and hence ${\rm Rad}(G_i)=1$ by \L o\'s' Theorem and the uniform definability of the soluble radical) the characteristically simple groups $M_{r,i}$ are products of simple non-abelian groups. As $G$ is an $\widetilde{\mathfrak{M}}_c$-group, and the groups $M_{r,i}$ are centerless and commute pairwise, we have by \cite[Lemma 4.1]{Zou2020} that the numbers $j_i$ are bounded and hence constant modulo $\mathcal{U}$. Similarly, each $M_{r,i}$ is a product of $\ell_i$ many isomorphic simple non-abelian groups $S_{r, k,i}$ where $\ell_i$ does not depend on $i$ modulo $\mathcal{U}$. Write $\ell$ for the constant value of $\ell_{i}$ (on a $\mathcal{U}$-large set). By semi-simplicity, $C_{{\rm Soc}(G_i)}(G_i)=1$, so $G_i \hookrightarrow {\rm Aut}({\rm Soc}(G_i))$ and the sizes $|{\rm Soc}(G_i)|$ grow without a bound when $i$ varies. This means that there are characteristically simple groups $ M_{r,i} \leqslant {\rm Soc}(G_i)$ so that the sizes $|M_{r,i} |$ grow without a bound when $i$ varies (in particular, the sizes of the simple components $S_{r, k,i}$ of $M_{r,i}$ grow without a bound).

Let $S_{r, 1,i}\times \cdots \times S_{r, \ell,i} = M_{r,i} \unlhd  {\rm Soc}(G_i)$, and $|S_{r, k,i}| \rightarrow \infty$ as $i$ varies, for all $k$. By a standard application of the Indecomposability Theorem (Theorem~\ref{th:indec} below) for groups of finite $SU$-rank, simple normal subgroups are definable and non-abelian definably simple groups are simple. So we get that \begin{itemize}
\item $M_r \equiv \prod_{i\in I} M_{r,i}/\mathcal{U}$ is definable, and
\item $M_{r} = S_{r,1} \times \cdots \times S_{r,\ell}$, with $S_{r,k}$ all definable isomorphic simple pseudo-finite groups with $ S_{r,k} \equiv \prod_{i\in I} S_{r,k,i}/\mathcal{U}$ for each $k\in \{1,\ldots, \ell\}$.
\end{itemize} 
The claim easily follows.\end{proof}

\subsection{Indecomposability and field interpretation results}

\begin{thm}[Indecomposability Theorem, Wagner {\cite[Theorem 5.4.5]{Wagner2000}}]\label{th:indec}Let $G$ be a supersimple group of finite $SU$-rank and $\{X_i : i\in I\}$ be a (possibly infinite) collection of definable subsets of $G$. Then there exists a definable $H \leqslant G$ so that \begin{enumerate}
\item $H \leqslant \langle X_i : i \in I \rangle$ and there are $i_0,\ldots, i_n\in I$ so that $H \leqslant X^{\pm 1}_{i_0}\cdots X^{\pm 1}_{i_n}$;
\item $X_i/H$ is finite for each $i\in I$.
\end{enumerate}
Moreover, if the collection $\{X_i : i\in I\}$ is setwise invariant under some group $\Sigma$ of definable automorphisms of $G$, then $H$ can be chosen to be $\Sigma$-invariant.\end{thm}

There are several field interpretation results in the broader context of finite-dimensional groups (see in particular \cite{Wagner2020, Elwesetal, deloro2024zilber}). The following version is slightly stronger than what we need, but worth stating as it seems that such a formulation does not appear in the literature.

We say that a subset $X$ of a group $A$ is \emph{almost $H$-invariant} for some group $H \leqslant {\rm Aut}(A)$ if there is some finite subset $H_0\subseteq H$ so that for any $h\in H$ we have $X^h\subseteq \bigcup_{h_0\in H_0}X^{h_0}$.

\begin{thm}[Field Interpretation]\label{th:field}Let $G$ be a pseudo-finite finite-dimensional group with fine and additive dimension and $A,H \leqslant G$ be infinite definable subgroups with $C_H(A)=1$, $A$ abelian and $H$ finite-by-abelian. Assume that $A$ has no proper definable almost $H$-invariant subgroup of infinite index. Then there is an interpretable pseudo-finite field $F$ so that $A\cong F^+$ and $H \hookrightarrow F^\times$.

If $H$ is assumed to be abelian then we may relax the latter assumption by removing `almost' from it.
\end{thm}

\begin{proof}Set $R:={\rm End}_H(A)$ to be the ring of endomorphisms of $A$ generated by $H$. Let $ r\in R \setminus \{0\}$. Since $H$ is (finite-by-)abelian, ${\rm ker}(r)$ and ${\rm im}(r)$ are (almost) $H$-invariant definable subgroups of $A$ and hence, by our assumptions, each of them is either trivial or a finite index subgroup of $A$. Now ${\ker}(r)$ cannot be a finite index subgroup of $A$ as otherwise ${\rm im}(r)$ is a finite (almost) $H$-invariant definable subgroup of $A$, thus trivial, whence we would get $r=0$. So any $r\in R \setminus \{0\}$ is injective; being a definable injective map from the pseudo-finite group $A$ to itself, $r$ is also surjective. So $R$ acts on $A$ by automorphisms. In particular it is invertible. By \cite[Proposition 3.6 or Corollary 3.10]{Wagner2020} the field of fractions $F$ of $R$ is an interpretable skew field; by Wedderburn’s Little Theorem $F$ is an intepretable pseudo-finite field. So $A \cong F^+$ and $H \hookrightarrow F^\times$.
\end{proof}

We say that a group $G$ is \emph{almost simple} if it has a normal subgroup $S$ such that $G \leq \mathrm{Aut}(S)$. Further, say that $G$ is \emph{definably almost simple} if the simple group $S$ is a definable subgroup of $G$.

\begin{lem} \label{lem: field aut} Suppose $G$ is a pseudo-finite definably almost simple group (in a language containing $\mathcal{L}_{gp}$) and let $S$ denote the definable simple subgroup, necessarily a (twisted) Chevalley group over a pseudo-finite field $\mathbb{F}$, such that 
    $$
    S \leq G \leq \mathrm{Aut}(S).
    $$
    Then the subgroup $G \cap \mathrm{Aut}(\mathbb{F})$ consisting of field automorphisms of $S$ is a definable subgroup of $G$ and the permutation group $(G \cap \mathrm{Aut}(\mathbb{F}),\mathbb{F})$ is interpretable in $G$.
\end{lem}

\begin{proof}As $S$ is a Chevalley group there is a Lie algebra $\mathbb{L}$ over $\mathbb{Z}$ and $S$ is isomorphic to a subgroup of $\mathrm{GL}(\mathbb{L}_{\mathbb{F}})$, where $\mathbb{L}_{\mathbb{F}} = \mathbb{F} \otimes_{\mathbb{Z}} \mathbb{L}$. By Ryten \cite[Chapter 5]{Ryten2007}, the field $\mathbb{F}$ is interpretable in $S$, hence in $G$. Also, the group $\mathrm{GL}(\mathbb{L}_{\mathbb{F}}) \subseteq \mathbb{F}^{n^{2}}$ where $n$ is the dimension of $\mathbb{L}_{\mathbb{F}}$, is interpretable in $\mathbb{F}$.  Moreover, there is an $\mathbb{F}$-definable subgroup $S'$ of $\mathrm{GL}(\mathbb{L}_{\mathbb{F}})$ and an $S$-definable isomorphism $\iota : S \to S'$. 

Ryten shows how to define a multiplication on a root subgroup $U_{r}(\mathbb{F}) \subseteq S$ such that it is isomorphic to $\mathbb{F}$.  So the field automorphisms are defined by the elements of $G$ that (a) preserve the (definable) root subgroup $U_{r}$ and (b) the associated automorphism on $S'$ is obtained by applying the induced automorphism on $\mathbb{F}$ coming from the action on $U_{r}$ coordinate-wise to the $n^{2}$ coordinates of the elements of $S'$. 

The argument for the case when $S$ is a twisted Chevalley group is the same, with the only difference being that Ryten shows that, in this case, there is some care in selecting the root subgroup. 
\end{proof}

\begin{fact}[{\cite[Corollary 5.5]{Zou2020}}]\label{fact:field-autom}Suppose $(F,B)= \prod_{i\in I} (\mathbb{F}_{p_i^{n_i}}, B_i)/\mathcal{U}$ is a pseudo-finite structure with $F$ a field and $B$ an infinite set of automorphisms of $F$. Then the theory of $(F,B)$ is not supersimple.\end{fact}

\subsection{Primitive permutation groups of bounded orbital diameter} \label{subsection: primfacts}

Let $(G,X)$ be a transitive permutation group. Then an \emph{orbital graph} for $(G,X)$ is a graph with vertex set $X$ whose edge set is an orbit of $G$ on $X^{\{2\}}$, the collection of unordered $2$-element subsets of $X$. Given $d\in \mathbb{N}$, denote by $\mathcal{F}_d$ the class of finite primitive permutation groups whose orbital graphs are of diameter at most $d$. A class $\mathcal{C}$ of finite primitive permutation groups is called \emph{bounded} if $\mathcal{C}  \subset \mathcal{F}_d$ for some $d$, that is, all the orbital graphs of members of $\mathcal{C}$ are of diameter at most $d$.


Note that if $(G,X)$ is a primitive (rather than merely definably primitive) permutation group equal to a nonprincipal ultraproduct $\prod (G_{i},X_{i})/\mathcal{U}$, where each $(G_{i},X_{i})$ is a finite primitive permutation group, then, modulo $\mathcal{U}$, the permutation groups $(G_{i},X_{i})$ belong to a bounded class, as explained in \cite{LMT}. 

We extract below a sufficient version (for our purposes) of the main theorem in \cite{LMT}.

\begin{thm}[Liebeck, Macpherson and Tent {\cite[Theorem 1.1 and discussion in Section 7]{LMT}}]\label{thm:LMT} Let $(G,X)$ be a supersimple primitive pseudo-finite permutation group of finite $SU$-rank, which is an ultraproduct of finite primitive permutation groups $(G_i, X_i)$. Then the following holds. 
\begin{enumerate}[(i)]

\item (Affine case) If ${\rm Rad}(G) \neq 1$, then $(G,X)$ is of the form $(VH, V)$, where $V=V_d(K)$ is a $d$-dimensional vector space over a pseudo-finite field $K$ and $H$ is an irreducible subgroup of ${\rm GL}_d(K)$.
\item If ${\rm Rad}(G) = 1$, then $(G,X)$ takes one of the following forms:
\begin{enumerate}
    \item (Almost simple case) $G$ is an ultraproduct of finite almost simple groups and $T = \mathrm{Soc}_{d}(G)$ is a non-abelian simple pseudo-finite group. 
    \item (Simple diagonal action case) $(G,X)$ is an ultraproduct of finite groups of simple diagonal action type and $\mathrm{Soc}_{d}(G) = T^{k}$ for some non-abelian simple pseudo-finite group $T$ and $k > 1$. 
    \item (Product action case) There is an $\ell$ and a primitive permutation group $(H,Y)$ definable in $(G,X)$ (and thus also pseudo-finite and of finite $SU$-rank) such that $G$ may be definably identified with a subgroup of $H \mathrm{wr} \mathrm{Sym}_{\ell}$ and $X$ may be identified with $Y^{\ell}$ with the action being the product action. Moreover, the permutation group $(H,Y)$ will fall into either the almost simple case or the simple diagonal action case (that is, is equal to an ultraproduct which is of almost simple or simple diagonal action type). 
\end{enumerate}
\end{enumerate}
Moreover, any $\omega$-saturated primitive pseudo-finite permutation group is elementarily equivalent to an ultraproduct of elements of a bounded class.
\end{thm}

\begin{proof}As is explained in \cite[Section 7]{LMT}, we have that $(G,X) = \prod_{i\in I}(G_i, X_i)/\mathcal{U}$, where $\mathcal{U}$-many of the finite groups $(G_i, X_i)$ belong to the same class (1)-(6) from \cite[Theorem 1.1]{LMT}. 

If ${\rm Rad}(G)\neq 1$ then $\mathcal{U}$-many groups $(G_i, X_i)$ are not semi-simple and thus they are of type (1); so one of the cases (a)-(c) in \cite[Section 7]{LMT} happens. Since the case (a) is precisely our statement (i), it is enough to note that the cases (b) and (c) cannot happen as infinite dimensional vector spaces over a pseudo-finite field are not of finite $SU$-rank.

Because we are assuming that $(G,X)$ is of finite $SU$-rank, our hypothesis excludes any ultraproducts of groups of unbounded Lie rank or of alternating groups of unbounded size in $\mathrm{Soc}_{d}(G)$ (Lemma~\ref{lem:socle}). Thus if $\mathrm{Rad}(G) = 1$, this leaves groups of types (ii)(a), (b), and (c).  

Now we turn to the product action case in some more detail.  To get definability of the permutation group $(H,Y)$ in $(G,X)$, we have to go into the proof of the O'Nan-Scott Theorem. We know (Lemma~\ref{lem:socle}) that, by finite $SU$-rank, $\mathrm{Soc}_{d}(G)$ is a definable subgroup of $G$ which is equal to $T_{1} \times \cdots \times T_{k}$ where the $T_{i}$ are definable isomorphic non-abelian simple groups.  Since we are interested in the product action case, we assume $k > 1$. Let $M = \mathrm{Soc}_{d}(G)$ and pick $x \in X$.  Let $\pi_{i} : M \to T_{i}$ denote the projection to the $i$th factor. In the proof of the O'Nan-Scott theorem, as given in \cite{liebeck1988nan}, product actions arise in two cases. 

The first case, $\pi_{i}(M_{x}) = T_{i}$ for some/all $i = 1, \ldots, k$. In this case, $M_{x}$ is a direct product $D_{1} \times \cdots \times D_{s}$ of full diagonal subgroups $D_{i}$ of subproducts $\prod_{j \in I_{i}} T_{j}$, where the $I_{i}$ partition $\{1, \ldots, k\}$. Assuming that the factors have been enumerated so that $I_{1} = \{1, \ldots, m\}$ (so $m \geq 2$), we set $K = T_{1} \times T_{m}$ and $N = N_{G}(K)$.  For a subgroup $L \leq N_{G}(K)$, $L^{*} = LC_{G}(K)/C_{G}(K)$ denotes the automorphisms of $K$ induced by $L$ via the conjugation action. Then the permutation group $(H,Y)$ is defined by setting $H = N^{*} = K^{*}N^{*}_{x}$ and $Y$ is the set of cosets of $N^{*}_{x}$ in $H$.  This gives a definition of the permutation group $(H,Y)$ in $(G,X)$ (using that the factors of the socle $T_{i}$ are definable).  

The second case is when $\pi_{i}(M_{x}) \lneq T_{i}$ for some/all $i$.  In this case, we set $N = N_{G}(T_{1})$ and, for a subgroup $L \leq N$, we define $L^{*} = LC_{G}(T_{1})/C_{G}(T_{1})$ to be the automorphisms of $T_{1}$ induced by $L$ via conjugation. Then, as in the first case, we define $H = N^{*}$ and $Y$ the set of cosets of $N^{*}_{x}$ in $H$.  This is also clearly definable. 
\end{proof}

\section{Primitivity and definable primitivity}\label{sec:prim}

\subsection{Structure of $(G,X)$} Using standard methods, we first describe the general structure of the definably primitive permutation groups that we study:

\begin{prop}\label{propo:socle}Suppose $(G,X)$ is a finite-dimensional definably primitive permutation group with fine and additive dimension. Then one of the following holds.
    \begin{enumerate}[(1)]
     \item ${\rm Rad}(G)\neq 1$. Then $G = A \rtimes G_{x}$, with ${\rm dim}(A)={\rm dim}(X)$ and the subgroup $A \unlhd G$ is either an elementary abelian $p$-group for some prime $p$ or a torsion free divisible abelian group. Further, $A = \mathrm{Soc}_{d}(G)=C_G(A)$ is the unique proper definable normal abelian subgroup of $G$.  
    
    \item ${\rm Rad}(G)= 1$. In this case, assume further that $(G,X)$ is pseudo-finite and of finite $SU$-rank. Then, either there is a unique minimal definable normal subgroup $M_1$ of $G$, or there are exactly $2$ minimal definable normal subgroups $M_1$ and $M_2=C_G(M_1)$ of $G$. In either of the two cases, for $r\in \{1,2\}$,
        $$
        M_{r} = S_{r,1} \times \cdots  \times S_{r,\ell_{r}}
        $$
        where $S_{r,k}$ are all isomorphic non-abelian simple pseudo-finite groups as $r$ varies and $k\in\{1, \ldots , \ell_r\}$. Also, $G$ acts transitively on $\{S_{r,1}, \ldots, S_{r,\ell_{r}}\}$. 
         \end{enumerate}
\end{prop}

\begin{proof} Assume first that ${\rm Rad}(G) \neq 1$. Then $G$ is not semi-simple and has a nontrivial abelian normal subgroup $H$. As is explained in Section~\ref{sec:semi-simple}, this implies that $G$ has a definable abelian normal subgroup $A$. Then, as $G$ acts definably primitively on $X$, $A$ acts transitively on $X$. Therefore, if $x^{a_1} =x
^{a_2}$ for some $x\in X$ and $a_1, a_2\in A$, then for any $y\in X$, $y=x^{a'}$ for some $a'\in A$. By commutativity of $A$, we get $$y^{a_1}=x^{a_1a'}=x^{a'a_1}=x^{a'a_2}=x^{a_2a'}=y^{a_2}.$$ Hence $a_1=a_2$ and $A$ acts regularly on $X$. So, given $x\in X$, the map sending $a\in A$ to $x^a$ is a definable bijection from $A$ to $X$. We have ${\rm dim}(A)={\rm dim}(X)$. So $A$ is the unique proper definable normal abelian subgroup of $G$. Therefore, for any $n\in \mathbb{N}$, either $nA=1$ or $nA=A$. If $nA=1$ for some $n$ then, by the minimality, $A$ is an elementary abelian $p$-group and if $nA=A$ for each $n$ then $A$ is divisible. Since, for any $n$, the set of all elements of order at most $n$ forms a characteristic and definable subgroup of $A$, in the latter case, $A$ is torsion-free. 

Now, note that $A \leqslant C_G(A)$ are definable normal subgroups of $G$. By repeating the argument above, both $A$ and $C_G(A)$ act regularly on $X$ so if we take $x \in X$ and $a' \in C_G(A)$, then there is some $a \in A$ such that $x^{a} = x^{a'}$, hence $a = a'$ so $a' \in A$ and $A=C_G(A)$. 

Now, let $x\in X$. As $A$ acts regularly on $X$, we have $A\cap G_x=1$. For any $g\in G$ there is a unique element $a\in A$ such that $x^a=x^g$. Hence, $x=x^{a^{-1}g}$ and $a^{-1}g \in G_x$. Thus $g\in A G_x$ and we get $G=A \rtimes G_x$. Let then $N$ be a proper definable normal subgroup of $G$. Then $N\cap A \neq 1$ so, by the above, $N\cap A=A$. So ${\rm Soc}_d(G)=A$.

We move on to observe (2). Assume then that $G$ is semi-simple, pseudo-finite and supersimple of finite $SU$-rank. 

By Lemma~\ref{lem:socle}, we know that ${\rm Soc}_d(G)\equiv \prod_{i\in I}{\rm Soc}(G_i) /\mathcal{U}$ is nontrivial and is a direct product of $j$ many minimal definable normal subgroups $M_r$ each of which is a direct product of finitely many (pseudo-)finite simple non-abelian groups. Thus, $G$ has at least one minimal normal definable subgroup.

We claim that $G$ has at most two minimal normal definable subgroups. If $G$ has two such subgroups $M_{1} \neq M_{2}$, then since by minimality $M_{1} \cap M_{2} = 1$ and both are normal, we have $\langle M_{1},M_{2} \rangle = M_{1} \times M_{2}$ and therefore $M_{2} \subseteq C_{G}(M_{1})$. Since $C_{G}(M_{1}) \unlhd G$, we know, by definable primitivity, that $C_{G}(M_{1})$ and $M_{1}$ both act regularly on $X$.  It follows, then, that $C_{G}(M_{1})$ is also a minimal definable normal subgroup of $G$, since any proper subgroup would act on $X$ intransitively. Therefore, $M_{2} = C_{G}(M_{1})$. Since $M_{2}$ was an arbitrary definable minimal normal subgroup distinct from $M_{1}$, we see that there cannot be three distinct definable minimal normal subgroups. 

Now we just need to note that the minimal normal definable subgroups of $G$ are infinite: by definable primitivity, nontrivial normal definable subgroups of $G$ act transitively on $X$. Since $X$ is infinite, this, in particular, implies that the minimal normal definable subgroups of $G$ are infinite. 

Finally note that for $r\in \{1,2\}$, the group $G$ acts transitively on $\{S_{r,1}, \ldots, S_{r,\ell_{r}}\}$ as $(G,X)$ is definably primitive.\end{proof}

\subsection{Primitivity}

The following has been observed under different assumptions in \cite[Proposition 6.1]{elwes2008measurable} and \cite[Lemma 2.6]{Elwesetal}.

\begin{lem} \label{lem:congruencefact}
    Suppose $(G,X)$ is a transitive supersimple permutation group with $X$ infinite. Define a relation $\sim$ on $X$ by $x \sim y$ if and only if $|G_{x} : G_{x} \cap G_{y}| < \infty$. Then $\sim$ is a definable $G$-invariant equivalence relation on $X$. 
\end{lem}

\begin{proof}Fix $x \in X$. By (the proof of) \cite[Lemma 4.2.6]{Wagner2000}, there is some $n_{x} < \omega$ such that, whenever $G_{x} \cap G_{y}$ has finite index in $G_{x}$, then it has index at most $n_{x}$.  By compactness, there must be some $n < \omega$ such that $n_{x} \leq n$ for all $x \in X$. This entails that $\sim$ is definable. The fact that $\sim$ is reflexive is clear and the argument that $\sim$ is transitive and $x \sim y$ entails $x^{g} \sim y^{g}$ for all $g \in G$ was shown in \cite[Proposition 6.1]{elwes2008measurable}.  It was pointed out in \cite[Lemma 2.6]{Elwesetal}, then, that NSOP (and therefore supersimplicity) entails that $\sim$ is symmetric. 
\end{proof}

\begin{lem}\label{lemma:primitive}Suppose $(G,X)$ is a supersimple pseudo-finite definably primitive permutation group of finite $SU$-rank. If $\widetilde{N}_G(G_x)=G_x$, then $(G,X)$ is primitive. \end{lem}
\begin{proof}Suppose for a contradiction that there is $K$ with $G_x < K < G$. Let $k\in K\setminus G_x$. Then $\widetilde{N}_G(G_x)=G_x$ implies that the index $|G_x:G_x^k\cap G_x|$ is infinite. 
By Theorem~\ref{th:indec}, there is a finite index subgroup $L$ of $H = \langle G_x^k : k\in K\rangle$ which is a definable normal subgroup of $K$. Since $L$ is normal in $K$ and both $G_{x}$ and $L$ are definable, the subgroup $G_{x}L$ is a definable subgroup of $K$ (and of $G$). Moreover, $G_{x}$ has infinite index in $H$ and $|H : G_{x}L|$ is finite, hence $G_{x}$ is of infinite index in $G_{x}L$. This shows $G_{x} < G_{x}L < G$, contradicting definable primitivity. So $(G,X)$ is primitive. \end{proof}

\begin{thm}[Smith \cite{Smith2015}]\label{th:finite-ps} Let $(G,X)$ be an infinite primitive permutation group with $G_x$ finite. Then $G$ is finitely generated.
\end{thm}

As mentioned in the introduction, the direction of the following theorem establishing that definable primitivity implies primitivity for permutation groups with large point stabilisers, has appeared in the literature with somewhat stronger hypotheses by Elwes and Ryten for measurable theories \cite[Proposition 6.1]{elwes2008measurable}, and Elwes, Jaligot, Macpherson, and Ryten in pseudo-finite finite $SU$-rank theories with elimination of the quantifier $\exists^{\infty}$ \cite[Lemma 2.6]{Elwesetal}. The other direction is new. 

\begin{thm}\label{th:primitive}Suppose $(G,X)$ is a supersimple pseudo-finite definably primitive permutation group of finite $SU$-rank. Then $(G,X)$ is primitive if and only if $G_x$ is infinite.
\end{thm}
\begin{proof} Assume first that the point stabiliser $G_x$ is finite and suppose, towards a contradiction, that $(G,X)$ is primitive. Theorem~\ref{th:finite-ps} implies that $G$ is finitely generated. We now invoke Proposition~\ref{propo:socle}(1): if ${\rm Rad}(G)\neq 1$ then $G$ is a finite extension of an abelian group $A$. Thus $A$ is a finitely generated pseudo-finite abelian group and hence finite by \cite[Proposition 3.3]{Point-Houcine2013}. This contradicts the infiniteness of $G$. The case ${\rm Rad}(G) =1$ works as in \cite[Theorem 1.2]{Karhumaki2021}: ${\rm Rad}(G)  =1$ implies that ${\rm Soc}_d(G)$ is a product of finitely many simple pseudo-finite (so, in particular, infinite) groups (this follows from Lemma~\ref{lem:socle} and from the fact that definable normal subgroups must be infinite as $(G,X)$ is definably primitive). Thus, each such simple component is normal in a finite index subgroup $G_0$ of $G$. But $G_0$ is finitely generated and thus, \cite[Lemma 2.2]{Palacin2018} implies that any of the components, say $S$, is finitely generated; being simple, such $S$ is finite by \cite[Proposition 3.14]{Point-Houcine2013}. This contradiction proves that $(G,X)$ is not primitive.

So assume that $G_x$ is infinite. By Lemma~\ref{lemma:primitive}, to show that $(G,X)$ is primitive, it is enough to show that $\widetilde{N}_G(G_x)=G_x$. 

Define $\sim$ on $X$ by $x \sim y \Leftrightarrow [G_x : G_x\cap G_y] < \infty$. By Lemma~\ref{lem:congruencefact} and the definable primitivity of $(G,X)$, either $\sim$ has exactly one class or all its classes are trivial. 

Suppose towards a contradiction that $\sim$-classes are not trivial. By definable primitivity, we know that either $\widetilde{N}_G(G_x)=G_x$ or $\widetilde{N}_G(G_x)=G$ for any $x\in X$. The former case implies primitivity. So we may assume that $G=\widetilde{N}_G(G_x)$ for each $x\in X$. Then all point-stabilisers $G_x$ are uniformly commensurable. Now Theorem~\ref{th:commensurable} implies that $G$ has a definable normal subgroup $N$ uniformly commensurable with each $G_x$. So the orbit of any $x \in X$ under $N$ must be finite. But the infinite group $N$, being a normal and definable subgroup of $G$, acts transitively on the infinite set $X$ (recall that $X$ is infinite as $G$ is infinite and acts on $X$ faithfully); a contradiction. 

By the above, we may assume that $\sim$-classes all have size $1$. Then, for each $x\in X$, $\widetilde{N}_G(G_x)=N_G(G_x)$. If $N_G(G_x)=G_x$ then we are done, so, by definable primitivity, we may assume that $N_G(G_x)=G$. Since $(G,X)$ is definably primitive and $1 \neq G_{x} \unlhd G$, we must have that $G_{x}$ acts on $X$ transitively, which is impossible since $G_{x}$ is the stabiliser of $x$. This contradiction completes the proof.\end{proof}

\begin{cor}\label{cor:bounded}Let $\mathcal{C}$ be a class of finite primitive permutation groups such that every
nonprincipal ultraproduct of members of $\mathcal{C}$ is definable in a structure with supersimple theory of finite $SU$-rank. Assume that for $(G,X)\in \mathcal{C}$ and $x\in X$, $|G_x|\rightarrow \infty$ as $|X| \rightarrow \infty$. Then $\mathcal{C}$ is a bounded class.  
\end{cor}
\begin{proof}The proof goes exactly as in \cite[Corollary 4.4]{LMT} with `measurable' replaced by `of finite $SU$-rank'.
\end{proof}

\begin{rem} In \cite[Introduction]{LMT}, it is asked whether there is an example of a primitive pseudo-finite permutation group such that in any/some $\omega$-saturated model of its theory, the group is not primitive. Theorem~\ref{th:primitive} in particular shows that in the finite $SU$-rank context such an example does not exist.\end{rem}

\begin{rem} There are pseudo-finite definably primitive permutation groups of finite $SU$-rank which are not primitive. Below we give an example which resembles the example by Macpherson and Pillay of a definably primitive permutation group of finite Morley rank which is not primitive \cite[p. 496]{Macpherson-Pillay1995}.

Let $H$ be a non-abelian finite group and let $I$ be an infinite set of primes not dividing $|H|$. For each $p \in I$, let $V_{p}$ be an irreducible $H$-representation, where $V_{p}$ is an $\mathbb{F}_{p}$-vector space of finite dimension $>1$.  We may assume that $\mathrm{dim}(V_{p})$ is the same for all $p \in I$ (by Maschke's theorem and pigeonhole principle). Notice that for each $p \in I$, the permutation group $( V_{p} \rtimes H,  V_{p})$ is primitive since the action is transitive and the stabiliser $(V_{p} \rtimes H)_0$ of $0 \in V_{p}$ is $H$, which is maximal by the irreducibility of the action of $H$ on $V_{p}$.  

Let $\mathcal{U}$ be a nonprincipal filter on $I$ and let $(V \rtimes H,V) = \prod_{p \in I} ( V_{p} \rtimes H,  V_{p})/\mathcal{U}$.  Because $(V \rtimes H,V)$ is an ultraproduct of primitive permutation groups, it is definably primitive. Also, $(V \rtimes H, V) $ is interpretable in the ultraproduct $\mathbb{F} = \prod_{i \in I} \mathbb{F}_{p}/\mathcal{U}$ so it is of finite $SU$-rank. However, $V$ is a vector space over $\mathbb{F}$ and $|\mathbb{F}| = 2^{\aleph_{0}}$ so $V$ is uncountable. Since $H$ is finite, there must be a countable $H$-invariant subgroup $U \leq V$. Then we have 
$$
(V_{p} \rtimes H)_0=H < UH < V \rtimes H,
$$
so $(V \rtimes H,V)$ is not primitive.
\end{rem}

It is worth noting that direction $\Leftarrow$ of Theorem~\ref{th:primitive} is known in the finite Morley rank setting, without pseudo-finiteness assumption \cite[Proposition 2.7]{Macpherson-Pillay1995}: \emph{if $(G, X)$ is a definably primitive permutation group of finite Morley rank with infinite point stabilisers, then $(G, X)$ is primitive.} There is no known infinite definably primitive permutation group of finite Morley rank $(G,X)$ with finite point stabilisers, which is primitive. Indeed, by Theorem~\ref{th:finite-ps}, the existence of such a group should yield the existence of a finitely generated infinite simple group of finite Morley rank, thus providing a counterexample to the famous Cherlin-Zilber conjecture. Therefore, we believe that part (a) in the question below has a positive answer.

The direction $\Rightarrow$ of Theorem~\ref{th:primitive} should be easy to generalise to the wider setting of finite-dimensional groups with fine and additive dimension, since a) this is the level of generality in Proposition~\ref{propo:socle}(1) and b) showing Lemma~\ref{lem:socle} in this setting should be straightforward (using results from finite group theory allowing us to drop the use of Theorem~\ref{th:indec}). Moreover, it is likely that, as many results on pseudo-finite groups of finite $SU$-rank, Theorem~\ref{th:primitive} can be fully generalised to this wider context. The challenge in the direction $\Leftarrow$ is the lack of an indecomposability result (Theorem~\ref{th:indec} does not generalise to this context).

\begin{quest}\label{q:prim} To what extent Theorem~\ref{th:primitive} can be generalised? In particular:\begin{enumerate}[(a)]
\item Given an infinite definably primitive permutation group $(G,X)$ of finite Morley rank, if the point stabilisers are finite, does it follow that $G$ is non-primitive?
\item Can Theorem~\ref{th:primitive} be proven without using Theorem~\ref{th:indec}, so that it holds for any finite-dimensional pseudo-finite definably primitive permutation group $(G,X)$ with fine and additive dimension? 
\end{enumerate}
\end{quest}

\section{Finding bounds}\label{sec:bounds}

Let $(G,X)$ be a pseudo-finite definably primitive permutation group of finite $SU$-rank. In the proof of Theorem~\ref{th: main}, we may assume $(G,X)$ is equal to a nonprincipal ultraproduct $\prod (G_{i},X_{i})/\mathcal{U}$ where each $(G_{i},X_{i})$ is a finite primitive permutation group. If $(G,X)$ has a finite point stabiliser $G_x$ then $SU(G)=SU(G/G_x)=SU(X)$. So, to prove Theorem~\ref{th: main}, we may assume that the point stabilisers $G_{x}$ are infinite. Thus, by Theorem \ref{th:primitive}, $(G,X)$ is primitive. It follows, then, that we may apply Theorem~\ref{thm:LMT}. So $(G,X)$ is of one of the types from Theorem~\ref{thm:LMT}:  \begin{enumerate}
\item an affine group,
\item an almost simple group,
\item of simple diagonal action type, or
\item of product action type.
\end{enumerate} A case-by-case analysis in this section gives the bounds of Theorem~\ref{th: main}. 

In what follows, when we say that a pseudo-finite group $(G,X)$ is of affine (resp. almost simple, simple diagonal action or product action) type it is assumed that the point stabilisers are infinite, that is, that $(G,X)$ is actually primitive and thus elementarily equivalent to a non-principal ultraproduct of permutation groups of affine (resp. almost simple, simple diagonal action or product action) type.



\subsection{Affine groups}

Our first observation is just an easy corollary of Proposition~\ref{propo:socle}:

\begin{lem}\label{lemma:rad-abelian}Suppose $(G,X)$ is a pseudo-finite finite-dimensional definably primitive permutation group with fine and additive dimension. Assume that ${\rm Rad}(G) \neq 1$. Then, if $G_x$ is abelian, then ${\rm dim}(G) \leq 2 {\rm dim}(X)$.  
\end{lem}

\begin{proof}The claim directly follows from Proposition~\ref{propo:socle} and Theorem~\ref{th:field}. It can also be easily observed without Theorem~\ref{th:field}: by Proposition~\ref{propo:socle}, we have $G=A\rtimes G_x$, with $C_G(A)=A$ and ${\rm dim}(X)={\rm dim}(A)$. Also, $A$ is the minimal definable $G_x$-invariant subgroup of $G$. Assume that $G_x$ is abelian. Let $a\in A$ be nontrivial. We have that $${\rm dim}(G)={\rm dim}(a^G)+{\rm dim}(C_G(a)),$$ $a^G=a^{G_x}$ and, by commutativity of $G_x$, if $y\in G_x \cap C_G(a)$, then $y\in C_G(a^G)=C_G(a^{G_x})$. But $Z(C_G(a^G))$ is a proper normal abelian definable subgroup of $G$, thus equal to $A$. So $y\in C_G(A) \cap G_x =A \cap G_x=1$. So $C_G(a)=A$ and $a^G \subseteq A$; whence ${\rm dim}(a^G) \leqslant {\rm dim}(A)$ and ${\rm dim}(G) \leqslant 2 {\rm dim}(A)=2 {\rm dim}(X)$.
\end{proof}

\begin{rem}Due to Theorem~\ref{thm:K-W}, Proposition~\ref{propo:socle}(1) and hence Lemma~\ref{lemma:rad-abelian} do not use the classification of finite simple groups. All the rest of our results heavily rely on CFSG.
\end{rem}

\begin{lem}\label{lem:affinecase}Suppose $(G,X)$ is a definably primitive pseudo-finite permutation group of finite $SU$-rank of affine type. Then if $r=SU(X)$, we have $SU(G) \leqslant r + (r^2+1)r$.

Further, if $r=1$ then $SU(G) = 2$ and there is an interpretable pseudo-finite field $F$ of $SU$-rank $1$, such that $(G,X)$ is definably isomorphic to $(F^+\rtimes G_x, F^+)$ where $G_x$ is a finite index subgroup of $ F^\times$.
\end{lem}

\begin{proof}Recall first that by Proposition~\ref{propo:socle}, since $(G,X)$ is of affine type, $G=A\rtimes G_x$ with $SU(A)=SU(X)$ and $A=C_G(A)$. Moreover, again since $(G,X)$ is of affine type, there is a pseudo-finite field $K$ so that $A\rtimes G_x \equiv V \rtimes H$, where $V$ is a $d$-dimensional vector-space over $K$ and $V\rtimes H$ is a subgroup of ${\rm AGL}_d(K)=V \rtimes {\rm GL}_d(K)$. Given any subset $V'\subseteq  V $ we have that $C_{V \rtimes H}(V')=V \rtimes C_H(V')$. Moreover, similarly as one shows that the centraliser dimension ${\rm cd}({\rm GL}_d(K))$ of ${\rm GL}_d(K)$ is at most $ d^2+1$ (see \cite[Proposition 2.1]{Myasnikov2004} and its proof) one can also show that the length of chains of centralisers of the form $C_{{\rm GL}_d(K)}(Y)$ where $Y \subseteq V$ is at most $d^2+1 \leqslant r^2+1$. So there are $r^2+1$ many elements $v_i\in V$ so that $C_{H}(V)=C_{H}(v_1, \ldots, v_{r^2+1})$. By \L os's theorem the same holds for $A$ and $G_x$. That is, $G_x$ acts on $A$ by conjugation, there are $r^2+1$ many elements $a_i\in A$ so that $C_{G_x}(A)=C_{G_x}(a_1, \ldots, a_{r^2+1})$, and, since $C_{G_x}(A)=1$, we have that the map $$G_x \rightarrow A^{r^2+1}; \,\, g_x \mapsto g_x(a_1)g_x(a_2)\cdots g_x(a_{r^2+1})$$ is a definable injection. So $G_x \subseteq A^{r^2+1}$ definably. We get that $SU(G_x)\leqslant (r^2+1)r$ and hence $SU(G)=SU(A)+SU(G_x) \leqslant r + (r^2+1)r$.


If $r=1$, then $H \leqslant {\rm GL}_1(K) \cong K^\times$ is abelian, and therefore $G_x$ is abelian too. Then, by Proposition~\ref{propo:socle} and Theorem~\ref{th:field}, there is a pseudo-finite field $F$ so that $A \cong F^\times$ and $G_x \leqslant F^\times$. So $SU(G_x) \leqslant SU(A)=1$, and since $G_x$ is infinite (as the affine type group $(G,X)$ is assumed to be primitive), we have $SU(G_x) = SU(A)=1$, $G_x$ is of finite index in $F^\times$, and $SU(G) = 2$.\end{proof}



\subsection{Almost simple groups}

In this section (and in Section~\ref{sec:trank1}) we consider Chevalley groups and twisted Chevalley groups over (pseudo-)finite fields. That is, groups of the form $G=X(F)$, where $$X\in \{A_n (n \geqslant 1), B_n (n \geqslant 2), C_n (n \geqslant 3), D_n (n \geqslant 4), E_6, E_7, E_8, F_4, G_2\}$$  or $$X\in \{^2A_n (n \geqslant 2), ^2D_n (n\geqslant 4),^2B_2, ^3D_4, ^2E_6, ^2F_4, ^2G_2\}.$$  We call $X$ the type of $G$ (it specifies both the Lie type and the Lie rank of $G$). A (twisted) Chevalley group is classical if $X\in \{A_n, B_n, C_n, D_n, ^2A_n, ^2D_n\}$. Standard references for (twisted) Chevalley groups are \cite{Carter1971, Steinberg1967} and some of the definability issues concerning Chevalley groups are considered in \cite{segal2023defining}. When we discuss automorphisms of (twisted) Chevalley groups, we follow the standard terminology of Steinberg \cite{Steinberg1967}; in particular, we do \emph{not} require graph automorphisms to be algebraic. 

\begin{defn}
    Suppose $(G,X)$ is a finite primitive permutation group with $G$ almost simple.  Suppose $S$ is the socle of $G$ and $H$ is a point stabiliser.  We say that $(G,X)$ is \emph{standard} if one of the following holds:
    \begin{enumerate}
        \item $S = {\rm Alt}_{n}$ and $X$ is an orbit of subsets or partitions of $\{1, \ldots, n\}$, 
        \item $S$ is a classical group with natural module $V$ and $X$ is an orbit of subspaces, or pairs of subspaces, or $S = \mathrm{Sp}_{n}(q)$, $q$ is even, and $H \cap S = O^{\pm}_{n}(q)$. 
    \end{enumerate}
    Otherwise, we say that $(G,X)$ is \emph{non-standard}.  We will refer to a pseudo-finite primitive permutation group as standard or non-standard if it is elementarily equivalent to an ultraproduct of standard or non-standard permutation groups respectively.
\end{defn}

Recall at this point (Remark~\ref{rem:alternating}) that when working in the finite $SU$-rank context, nonprincipal ultraproducts of alternating groups do not appear definably. Below, this allows us to reduce our treatment of the almost simple case to the case of standard actions where the $d$-socle is a classical (twisted) Chevalley group.

\begin{thm}[Burness, Liebeck and Shalev {\cite[Corollary 1]{Burnessetal2009}}] \label{fact: small base}Let $(G,X)$ be a finite almost simple group in a primitive faithful non-standard action. Assume that $G$ is not the Mathieu group $M_{24}$. Then there is a base of size $6$, that is, there is a subset $B \subseteq X$ with $|B| = 6$ such that $\bigcap_{b \in B} G_{b} = 1$. 

\end{thm}

\begin{lem} \label{lem: nonstandard bound}Suppose $(G,X)$ is a pseudo-finite primitive permutation group of almost simple type of finite SU-rank. If $(G,X)$ is non-standard, then $SU(G) \leq 6SU(X)$. 
\end{lem}

\begin{proof}
    As the infinite group $(G,X)$ is elementarily equivalent to an infinite ultraproduct of finite permutation groups of non-standard almost simple type, Theorem \ref{fact: small base} and {\L}os's Theorem entails that $(G,X)$ also has a base of size $\leq 6$. Let $x_{1}, \ldots, x_{6} \in X$ be a base.  Then the map $G \to X^{6}$ defined by $g \mapsto (g \cdot x_{1}, \ldots, g \cdot x_{6})$ is injective, so we obtain 
    $$
    SU(G) \leq SU(X^{6}) = 6 \cdot SU(X), 
    $$
    as desired.
\end{proof}

\begin{lem} \label{lem: red to simple}
    Suppose $(G,X)$ is a pseudo-finite primitive permutation group of finite SU-rank of almost simple type.  Then $\mathrm{Soc}_{d}(G)$ has finite index in $G$ and acts transitively on $X$.
\end{lem}

\begin{proof}
    By Theorem \ref{th:wilson} (and Lemma~\ref{lem:socle}), $S = \mathrm{Soc}_{d}(G)$ is a (twisted) Chevalley group over a pseudo-finite field $F$.  By \cite[Theorem 30, Theorem 36]{Steinberg1967}, any automorphism of $S$ can be factored as a product of an inner automorphism, a graph automorphism, a diagonal automorphism, and a field automorphism.  There are at most finitely many graph and diagonal automorphisms and, by Lemma \ref{lem: field aut} and Fact \ref{fact:field-autom}, there can be only finitely many field automorphisms that appear in the factorisation of the elements of $G$, by our assumption that $G$ has finite $SU$-rank.  This shows $G/S$ is finite and, since $S$ is normal and $(G,X)$ is primitive, we must have that $(S,X)$ is transitive.
\end{proof}

Lemma \ref{lem: nonstandard bound} shows that our task of finding a bound for $SU(G)$ in terms of $SU(X)$ is now reduced, in the almost simple case, to the case of standard actions.  Then Lemma \ref{lem: red to simple} shows that this further reduces to standard actions of a \emph{simple} classical group $G$ acting transitively on $X$.  We will handle this case by a further case division between the cases of $\mathrm{Soc}_{d}(G)$ being a Chevalley group and a twisted Chevalley group.

\subsubsection{Chevalley groups}

The aim of this subsection is to consider permutation groups of the form $(G,G/P)$ where $G$ is a simple Chevalley group of classical type and $P \leq G$ is a parabolic subgroup.  

We will need some facts about Weyl groups.  If $T \leq G$ is a maximal torus, the Weyl group of $G$ is the group $W = N_G(T)/T$ and it can be presented as a group generated by reflections $s_{\alpha}$ for $\alpha \in \Pi$, where $\Pi$ denotes a set of simple roots in the root system $\Phi$ of $G$. Given $I \subseteq \Pi$, we can define the parabolic subgroup $W_{I} = \langle s_{\alpha} : \alpha \in I\rangle$ of the Weyl group $W$ generated the reflections associated to the roots in $I$. We set $W^{I} = \{w \in W : \ell(ws_{\alpha}) > \ell(w) \text{ for all }\alpha \in I\}$, where for any $w\in W$, $\ell(w)$ denotes the minimal length of expressions of $w$ as a product of fundamental reflections.  Then the multiplication map $W^{I} \times W_{I} \to W$ is a bijection and each element of $W^{I}$ is the minimal length representative of its $W_{I}$ coset in $W$ \cite[Proposition 2.4.4]{bjorner2005combinatorics}.  

We will begin by proving a lemma relating dimension and $SU$-rank for certain subgroups of $G$ and for the partial flag variety $G/P$ for $P \leq G$ parabolic. If $G$ is a simple Chevalley group over a pseudo-finite field $F$, then $F$ and $G$ are bi-interpretable by the second part of Theorem \ref{th:wilson}.  Thus it makes sense, in the situation where $G$ is considered as a structure in the language of groups (or any larger language) to talk about the $SU$-rank of $F$.  In what follows, when we refer to the SU-rank of $F$, we are identifying $F$ with the field obtained via the interpretation defined by Ryten in \cite{Ryten2007}. Using the field to find coordinates for $G$ and $G/P$, we will find expressions for the $SU$-rank of $G$ and of $G/P$ for $P \leq G$ parabolic in terms of the $SU$-rank of $F$ and the algebraic dimension of $G$ and $G/P$. This allows us to use the field to interpolate between the $SU$-rank of the set $G/P$ on which $G$ acts, and $G$ itself, in order to get the desired bounds.

In this section (and in Section~\ref{sec:trank1}), if $Y$ is an algebraic variety, we write $\mathrm{dim}(Y)$ for the dimension of $Y$ as an algebraic variety.  

\begin{lem} \label{lem: dimension inequality}
Suppose $G$ is a simple Chevalley group over a pseudo-finite field $F$ of finite SU-rank. Then we have the following:
\begin{enumerate}
    \item There is a $\mathcal{L}_{gp}$-definable maximal torus $T \leq G$ with $SU(T) = r \cdot SU(F)$, where $r = \mathrm{dim}(T)$ is the Lie rank of $G$.
    \item $SU(G) = \mathrm{dim}(G) \cdot SU(F)$.
    \item If $P \leq G$ is a parabolic subgroup, then $P$ is $\mathcal{L}_{gp}$-definable and $SU(P) = \mathrm{dim}(P) \cdot SU(F)$. 
\end{enumerate}  
\end{lem}

\begin{proof}
(1) The construction of a maximal torus $T$ given in, e.g., \cite[Lemma 28]{Steinberg1967} describes $T$ as the image of the surjection $\mathbb{G}_{m}^{d} \to T$  with finite kernel defined by $(t_{1}, \ldots, t_{d}) \mapsto \prod_{i=1}^{d} h_{i}(t_{i})$ for algebraic morphisms $h_{i}$ defined over the prime subfield and $d$ equal to the dimension of the maximal torus. This shows $SU(T) = d \cdot SU(F)$.  Definability follows from \cite{Ryten2007}. 

(2) Consider $B$ a Borel subgroup of $G$ containing $T$. The Bruhat decomposition of $G$ expresses $G$ as a disjoint union of double cosets $B \dot{w}B$ where $\dot{w}$ is an element of $N_G(T)$ corresponding to the Weyl group element $w \in W = N_G(T)/T$ \cite[Section 8.3]{springer1998linear}.  If $\mathbf{s} = s_{1}s_{2}\ldots s_{h}$ is a reduced expression for $w$ with $s_{i} = s_{\alpha_{i}}$ for $\alpha_{i}$ simple reflections then the map $\phi(x_{1},\ldots, x_{h};b) = u_{\alpha_{1}}(x_{1})\dot{s_{1}} u_{\alpha_{2}}(x_{2})\dot{s_{2}} \ldots u_{\alpha_{h}}(x_{h})\dot{s}_{h} b$ defines a bijection $\mathbb{A}^{h} \times B \to B \dot{w} B$ \cite[8.3.6]{springer1998linear}.  Therefore, we have 
$$
SU(G) = \max_{w \in W} SU(B\dot{w}B) = \max_{w \in W} \left( SU(B) + \ell(w)SU(F) \right). 
$$
The Borel subgroup may be expressed as a semi-direct product $U \rtimes T$ where $U = \prod_{r \in \Phi^{+}} U_{r}$ \cite[Section 8.5]{Carter1971}, where each $U_{r}$ is a root subgroup, definably isomorphic to $(F,+)$. This shows that $SU(B) = \mathrm{dim}(B) \cdot SU(F)$. Since $\mathrm{dim}(G)=\mathrm{dim}(B)+ \ell(w_0)$, where $w_0$ is the longest word in $W$, we obtain $SU(G) = \mathrm{dim}(G) \cdot SU(F)$.

(3)  If $P \leq G$ is parabolic, then it contains a Borel subgroup $B$.  Then $P = \bigcup_{w \in W_{J}} B \dot{w} B$ for a parabolic subgroup $W_{J} \subseteq W$. Using the Bruhat decomposition as above, we denote by $\Omega_{w}$ the image of the double coset $B \dot{w} B$ in $G/P$ for each $w \in W^{J}$.  By \cite[Proposition 3.1.2(ii)]{sweeting_decomposing}, $\Omega_{w}$ is in (definable) bijection with a product $\prod_{r \in \Phi^{-} \cap w^{-1}(\Phi^{+})} U_{r}$ of root subgroups and the length of this product is equal to the length of $w$.  This length is maximised at the word $w_{0}^{J} \in W^{J}$, where $w_{0} \in W$ is the longest word. Each subgroup $U_{r}$ satisfies $SU(U_{r}) = SU(F)$ so we get 
$$
SU(G/P) = \ell(w^{J}_{0}) \cdot SU(F) = \mathrm{dim}(G/P) \cdot SU(F),
$$
as desired.
\end{proof}

The next fact seems to be folklore.  The characteristic zero case can be found in \cite[Remark 1(iv)]{flaschka1991torus}. The following argument was pointed out to us by Sam Evens:

\begin{fact} \label{fact: sam fact}
Suppose $G$ is a simple algebraic group with parabolic subgroup $P$ and maximal torus $T$. Then $\mathrm{dim}(G/P) \geq \mathrm{dim}(T)$.  
\end{fact}

\begin{proof}
    We can fix a Borel subgroup $B \leq G$ and we may assume $T \leq B \leq P$. The Bruhat decomposition expresses $G$ as a disjoint union $B\dot{w}B$ as $w$ ranges over the Weyl group of $G$.  The cell $B\dot{w_{0}}B$, where $w_{0}$ is the longest word in the Weyl group, is the `big cell', which has open dense image $\Omega_{w_{0}} \subseteq G/P$ and the image $\Omega_{w_{0}}$ may be identified with a product $\prod_{\alpha \in I} U_{\alpha}$ of root subgroups of $G$ for some set $I$ of roots \cite{sweeting_decomposing}. 
    
    Say a point $\bar{x}=(x_{\alpha})_{\alpha \in I} \in \prod_{\alpha \in I} U_{\alpha}$ is \emph{general} if $x_{\alpha} \neq u_{\alpha}(0)$ for all $\alpha \in I$, where $u_{\alpha} : \mathbb{G}_{a} \to G$ is the homomorphism which is an isomorphism onto its image $U_{\alpha}$.  Note that if $t \in T$ and $x_{\alpha} = u_{\alpha}(c)$, then $tx_{\alpha}t^{-1} = t u_{\alpha}(c) t^{-1} = u_{\alpha}(\alpha(t)c)$ so if $tx_{\alpha} t^{-1} = x_{\alpha}$ then $c = \alpha(t) c$ so we have $\alpha(t) = 1$, since $c \neq 0$ (by the choice of $\overline{x}$).  Let $S = T_{\overline{x}}$, the stabiliser of $\overline{x}$ under the action of $T$.  Then $S$ is equal to the closed subgroup $\bigcap_{\alpha \in I} \mathrm{ker}(\alpha)$ and therefore does not depend on the choice of a general point $\overline{x}$. It follows that $S$ acts trivially on the open set consisting of all general points, and therefore acts trivially on all of $G/P$. 

    By simplicity, $G$ acts faithfully on $G/P$ and hence $T$ acts on $G/P$ with trivial kernel. Thus $S=1$ and there is an orbit of $T$ with dimension equal to $\mathrm{dim}(T)$, so $\mathrm{dim}(G/P) \geq \mathrm{dim}(T)$, as desired. 
\end{proof}

\begin{fact} \label{fact: dim bound 2}
    Suppose $G$ is a simple algebraic group not of exceptional type with rank $r$ (i.e. the dimension of a maximal torus in $G$ is $r$).  Then $\mathrm{dim}(G) \leq 2r^{2} + r$.  
    
\end{fact}

\begin{proof}
    This is well-known, but part of the proof of Lemma \ref{lem: dimension inequality} was rehearsing the proof that the dimension of $G$ is equal to the sum of the Borel subgroup plus the length of the longest reduced word in the Weyl group.  The Borel subgroup is the semi-direct product of a unipotent group $U$, whose dimension is equal to the number of positive roots in the root system of $G$, and $T$, the maximal torus whose dimension is assumed to be $r$. The length of the longest reduced word in the Weyl group is also the number of positive roots, so $\mathrm{dim}(G) = 2|\Phi^{+}| + r$.  From the table in \cite[Section 12.2]{humphreys2012introduction}, one sees this is maximised when the number of positive roots is $r^{2}$ (which occurs in types $B_{r}$ and $C_{r}$).
\end{proof}

\begin{rem}
    A careful reader will note that we are assuming $G$ is abstractly simple, rather than adopting the usual meaning of simple algebraic group, which allows for a finite kernel.  This is the only situation we have to deal with, since the groups appearing in the $d$-socle are abstractly simple, but the above proof works (with minor modifications) to establish the bound even for simple algebraic groups (in the conventional meaning for an algebraic group). 
    
    The groups of type $E_{7}$, $E_{8}$, $F_{4}$, and $G_{2}$ have dimensions $133$, $248$, $52$, and $14$ respectively, which violate the above bound.  The dimension of $E_{6}$ just makes it.  As these are the only exceptional types, it is easy to adjust the bound to additionally give an upper bound on $SU(G/P)$ in terms of $SU(G)$ for groups $G$ of exceptional type (for our purposes such bounds are not needed since the socle of the standard permutation group $(G,X)$ is of classical type).
\end{rem}

\begin{prop} \label{prop: Chevalley bound}
    Suppose $G$ is a simple Chevalley group over a pseudo-finite field $F$.  If $G$ has finite SU-rank and $P \leq G$ is a parabolic subgroup with $SU(G/P) = n$, then $SU(G) \leq 2n^{2} + n$. 
\end{prop}

\begin{proof}
    We know, by Lemma \ref{lem: dimension inequality}(3), that
    $$
    n = SU(G/P) = \mathrm{dim}(G/P) \cdot SU(F).
    $$
    Let $r$ denote the rank of $G$ (dimension of a maximal torus). Then by Fact \ref{fact: sam fact}, we have $r \leq  \frac{n}{SU(F)}$. Then, by Fact \ref{fact: dim bound 2}, we have 
    $$
    \mathrm{dim}(G) \leq 2\left( \frac{n}{SU(F)} \right)^{2}  + \frac{n}{SU(F)}. 
    $$
    Lemma \ref{lem: dimension inequality}(2) now gives us 
    $$
   SU(G) = \mathrm{dim}(G) \cdot SU(F) \leq \frac{2n^{2}}{SU(F)} + n \leq 2n^{2} + n,
    $$
    since $SU(F) \geq 1$. \end{proof}

\subsubsection{Twisted Chevalley groups}

In this subsection, we will consider permutation groups of the form $(G_{\sigma}, G_{\sigma}/P)$ where $G_{\sigma}$ is a pseudo-finite twisted simple Chevalley group that is also a classical group (e.g. a unitary group of type $^2A_n$ for $n \geqslant 2$, or an orthogonal group of type $^2 D_n$ for $n \geqslant 4$) and $P \leq G_{\sigma}$ is a parabolic subgroup. As the notation suggests, $G_{\sigma}$ is obtained from a Chevalley group $G$ by taking the subgroup fixed under an endomorphism $\sigma$.  Let $\theta$ denote the associated field automorphism.  The case that $G_{\sigma}$ is classical over a field $F$ entails that $\theta^{2} = 1$ and $F$ has a quadratic extension $K$ with $\mathrm{Gal}(K/F) = \langle \theta \rangle$ \cite[Section 14.5]{Carter1971}. Moreover, by the second part of Theorem~\ref{th:wilson} if $G_{\sigma}$ is twisted Chevalley group over a pseudo-finite field $F$, then $G_{\sigma}$ is bi-interpretable with the field $F$. Note that the quadratic extension $K/F$ is also a pseudo-finite field so we may regard $G_{\sigma}$ a subgroup of the pseudo-finite Chevalley group $G(K)$. 

We will write $\Phi$ for the root system of $G$ with a choice $\Pi$ of simple roots and $\Phi^{+}$ and $\Phi^{-}$ of positive and negative roots, respectively.  For each $\alpha \in \Phi$, we denote by $U_{\alpha}$ the associated root subgroup. For an element $w \in W$, let $U^{-}_{w} = \prod_{\alpha \in \Phi^{+} \cap w^{-1}(\Phi^{-})} U_{\alpha}$ and $V^{+}_{w} = \prod_{\alpha \in \Phi^{-} \cap w(\Phi^{+})} U_{\alpha}$. 

We may fix a $\sigma$-stable maximal torus $T \subseteq G(K)$ and a $\sigma$-stable Borel subgroup $B \subseteq G(K)$ so that, setting $N = N_{G(K)}(T)$, we have that their $\sigma$-fixed subgroups $B_{\sigma} \subseteq B$ and $N_{\sigma} \subseteq N$ form a $(B,N)$-pair for the group $G_{\sigma}$ \cite[Theorem 13.5.4]{Carter1971}. We have that $\sigma$ acts on the Weyl group $W = N/(B \cap N)$ and the  $\sigma$-fixed subgroup $W_{\sigma}$ is the Weyl group of $G_{\sigma}$.  We may write $W_{\sigma} = \langle s_{I} : I \in S \rangle$ where $S$ is the set of $\sigma$ orbits on the set of simple roots $\Pi$ and each $s_{I}$ is the longest reduced word in the parabolic subgroup $W_{I}$ of $W$, which is an involution \cite[Lemma 23.3]{Malle-Testerman}.  As above, for each $J \subseteq S$, we will write $(W_{\sigma})_{J}$ for the parabolic subgroup of $W_{\sigma}$ generated by the elements $s_{I}$ for $I \in J$. As in the previous section, given $J \subseteq S$, we set $(W_{\sigma})^{J} = \{w \in W_{\sigma} : \ell(ws_{I}) > \ell(w) \text{ for all }I \in J\}$. Then the multiplication map $W_{\sigma}^{J} \times (W_{\sigma})_{J} \to W_{\sigma}$ is a bijection and each element of $W_{\sigma}^{J}$ is the minimal length representative of its $(W_{\sigma})_{J}$ coset in $W_{\sigma}$ \cite[Proposition 2.4.4]{bjorner2005combinatorics}.

For each Weyl group element $w \in W$, choose some representative $\dot{w} \in G$; this can be done in such a way that $\dot{w} \in G_{\sigma}$ for all $w \in W_{\sigma}$. Let $w_{0}$ denote the longest word in $W$.  Then $w_{0} \in W_{\sigma}$ \cite[Lemma C.2]{Malle-Testerman}.

\begin{lem} \label{lem: parameterization}
    Fix $w \in W_{\sigma}$. There is a definable bijection $B_{\sigma} \times (V^{+}_{w})_{\sigma} \to B_{\sigma} \dot{w} B_{\sigma}$ given by $(b,v) \mapsto bv \dot{w}$. 
\end{lem}

\begin{proof}
    The usual Bruhat decomposition for twisted Chevalley groups \cite[Proposition 13.5.3]{Carter1971} establishes a bijection $B_{\sigma} \times (U^{-}_{w})_{\sigma} \to B_{\sigma} \dot{w} B_{\sigma}$ via the map $(b,u) \mapsto b \dot{w} u$.  Note that we have 
    \begin{eqnarray*}
    B_{\sigma} \dot{w} (U^{-}_{w})_{\sigma} &=& B_{\sigma} \dot{w} (U^{-}_{w})\dot{w}^{-1} \dot{w} \\
    &=& B_{\sigma} \left(\prod_{\alpha \in \Phi^{+} \cap w^{-1}(\Phi^{-})} U_{w(\alpha)} \right)_{\sigma} \dot{w} \\
    &=& B_{\sigma} \left( \prod_{\alpha \in \Phi^{-} \cap w(\Phi^{+})} U_{\alpha} \right)_{\sigma} \dot{w} \\
    &=& B_{\sigma} (V^{+}_{w})_{\sigma} \dot{w},
    \end{eqnarray*}
    and the result follows. 
\end{proof}

\begin{lem} \label{lem: root space iso}
    Suppose $J \subseteq S$ and $P_{J} = \bigsqcup_{w \in (W_{\sigma})_{J}} B_{\sigma} \dot{w} B_{\sigma}$ is the associated standard parabolic subgroup of $G_{\sigma}$. Let $\Omega^{J}_{w}$ be the image of the double coset $B_\sigma\dot{w}B_{\sigma} $ in $G_\sigma/P_{J}$ for $w\in W^{J}_{\sigma}$. Write $w_{0} = w^{J}w_{J}$ with $w^{J} \in W_{\sigma}^{J}$ and $w_{J} \in (W_{\sigma})_{J}$. Then 
    $$
    SU(G_{\sigma}/P_{J}) = SU(\Omega^{J}_{w^{J}}) = SU((V^{+}_{w^{J}})_{\sigma}).
    $$
\end{lem}

\begin{proof}
    First, using that, whenever $s \in S$ and $l(sw) > l(w)$, then $B_{\sigma}\dot{s}B_{\sigma} B_{\sigma}\dot{w}B_{\sigma} = B_{\sigma}\dot{s}\dot{w} B_{\sigma}$ (a property of $(B,N)$-pairs), we note that we can decompose $G_{\sigma}$ into a disjoint union as follows:
    $$
    G_{\sigma} = \bigsqcup_{w \in W_{\sigma}} B_{\sigma} \dot{w} B_{\sigma} = \bigsqcup_{w \in W_{\sigma}^{J}} \bigsqcup_{w' \in (W_{\sigma})_{J}} B_{\sigma} \dot{w} B_{\sigma} B_{\sigma} \dot{w}' B_{\sigma} = \bigsqcup_{w \in W^{J}_{\sigma}} B_{\sigma} \dot{w} B_{\sigma} P_{J}
    $$
    and therefore $G_{\sigma}/P_{J}$ is the disjoint union of the $\Omega^{J}_{w}$ as $w$ ranges over $W^{J}_{\sigma}$.  By Lemma \ref{lem: parameterization} and the fact that $P_{J} \supseteq B_{\sigma}$, we know that the map $v \mapsto P_{J} v \dot{w}$ is a surjection from $(V^{+}_{w})_{\sigma}$ to $G_{\sigma}/P_{J}$ for all $w \in W_{\sigma}^{J}$.  Moreover, it is injective since $P_{J}$ intersects the $U_{\alpha}$ for $\alpha \in \Phi^{-}\cap w^{-1}(\Phi^{+})$ trivially \cite[3.2.2(ii)]{sweeting_decomposing}. As the longest word moves all the positive roots to negative roots, the $(V^{+}_{w})_{\sigma}$ of maximal $SU$-rank corresponds to the image of the double coset $B_{\sigma} \dot{w_{0}} B_{\sigma}$ since $w_{0}$ is the longest word.  This gives 
    $$
    SU(G_{\sigma}/P_{J}) = SU(\Omega^{J}_{w^{J}}) = SU((V^{+}_{w^{J}})_{\sigma}).
    $$
\end{proof}

\begin{lem} \label{lem: twisted roots}
    Let $A$ denote a $\sigma$-orbit in $\Phi$ and write $U_{A}$ for $\prod_{\alpha \in A} U_{\alpha}$.  Then 
    $$
SU(U_{A}) = 2 SU((U_{A})_{\sigma}). 
    $$
\end{lem}

\begin{proof}
    We apply \cite[Lemma 63]{Steinberg1967}, which describes the twisted root subgroups in the twisted Chevalley groups $G_{\sigma}$. Our assumption that $G_{\sigma}$ is classical entails that $\sigma$ is an involution and so there are only $3$ possibilities for $A$.  If $A$ is of type $A_{1}$, say $A = \{\alpha\}$, then $(U_{A})_{\sigma} = \{x_{\alpha}(t) : t \in F\}$ and $U_{A} = U_{\alpha} = \{x_{\alpha}(t) : t \in K\}$.  Then $SU(U_{A}) = SU(K) = 2SU(F) = 2SU((U_{A})_{\sigma})$, since $K/F$ is quadratic. 

    If $A$ is of type $A_{1} \times A_{1}$, $A = \{\alpha, \sigma(\alpha)\}$, say, then 
    $$
    (U_{A})_{\sigma} = \{x \cdot \sigma(x) : x = x_{\alpha}(t), \alpha \in A, t \in K\},
    $$
    and $U_{A} = U_{\alpha} \times U_{\sigma(\alpha)}$ so we obtain 
    $$
    SU(U_{A}) = 2 SU(K) = 2 SU((U_{A})_{\sigma}).
    $$

    If $A$ is of type $A_{2}$, then $A = \{\alpha, \beta, \alpha+\beta\}$, and we have 
    $$
    (U_{A})_{\sigma} = \{x_{\alpha}(t)x_{\beta}(\theta(t))x_{\alpha+\beta}(u) : t,u \in K, t\theta(t) + u + \theta(u) = 0\}.
    $$
    Then $SU((U_{A})_{\sigma}) = 3SU(F)$. To see this, note that if $\mathrm{char}(F) \neq 2$, then we can write $K = F(\gamma)$ for some element $\gamma \in K$ with $\gamma^{2} \in F$ and $\theta(\gamma) = - \gamma$. Then if $t,u \in K$, we may write $t = a + b\gamma$ and $u = c + d \gamma$ for $c,d \in F$. The equation $t\theta(t) + u + \theta(u) = 0$ yields the equivalent equation 
    $$
    (a+b\gamma)(a - b\gamma) + (c+d\gamma) + (c - d\gamma) = a^{2} - b^{2}\gamma^{2} + 2c = 0.
    $$
    Then it is clear that $a,b,d \in F$ can be arbitrary and, having chosen them, $c$ is completely determined. This gives $SU((U_{A})_{\sigma}) = 3 SU(F)$. 
    
    If, on the other hand, $\mathrm{char}(F) = 2$, then we may write $K = F(\gamma)$ for some Artin-Schreier root $\gamma$, so $\gamma^{2} + \gamma \in F$ and $\theta(\gamma) = \gamma + 1$. Then, writing $t = a+b\gamma$ and $u = c + d\gamma$ as above, the equation $t\theta(t) + u + \theta(u) = 0$ yields the equivalent equation 
    $$
    (a + b\gamma)(a + b\gamma + b) + (c + d\gamma) + (c + d\gamma + d) = a^{2} + b^{2}(\gamma^{2}+\gamma) + ab + d = 0.
    $$
    so $a,b,c \in F$ can be arbitrary and, having chosen them, $d \in F$ is completely determined. This likewise gives $SU((U_{A})_{\sigma}) = 3 SU(F)$. 
    
    Since $|A| = 3$, we have $SU(U_{A}) = 3SU(K)$.  Since $SU(K) = 2SU(F)$, this gives the desired conclusion. 
    \end{proof}

\begin{lem} \label{lem: lifting}
    Let $P_{J} \leq G_{\sigma}$ be a standard parabolic subgroup and let $I \subseteq \Pi$ be the union of the elements of $J$ (which, recall, are $\sigma$-orbits). Let $P^{*}_{I}$ denote the associated standard parabolic subgroup of $G$. Then 
    $$
    SU(G/P^{*}_{I}) \leq 2 \cdot SU(G_{\sigma}/P_{J}).
    $$
    
\end{lem}

\begin{proof}
Write the longest element $w_{0} \in W$ as $w^{I} \cdot w_{I}$ with $w^{I} \in W^{I}$ and $w_{I} \in W_{I}$.  Note that, since $w_{J} \in (W_{I})_{\sigma} = (W_{\sigma})_{J}$, we have that $w^{I}$ and $w^{J}$ are in the same coset of $W_{I}$.  Since $w^{I}$ is the shortest length representative of its $W_{I}$ coset in $W$, we necessarily have $\ell(w^{I}) \leq \ell(w^{J})$ and therefore 
$$
SU(V^{+}_{w^{I}}) \leq SU(V^{+}_{w^{J}}).
$$
By \cite[Lemma 3.1.2(ii)]{sweeting_decomposing} (as in Lemma \ref{lem: root space iso}), $G/P^{*}_{I}$ is in definable bijection with $V^{+}_{w^{I}}$. Now $(V^{+}_{w^{J}})_{\sigma}$ can be expressed as a product of root subgroups of $G_{\sigma}$ of the form $(U_{A})_{\sigma}$ for $A$ a $\sigma$-orbit of roots, so, by Lemmas \ref{lem: twisted roots} and \ref{lem: root space iso}, we have 
$$
SU(G/P^{*}_{I}) \leq SU(V^{+}_{w^{J}}) \leq 2 \cdot SU((V^{+}_{w^{J}})_{\sigma}) = 2 \cdot SU(G_{\sigma}/P_{J}).
$$\end{proof}

\begin{prop}\label{prop: bound twisted}
    Suppose $n = SU(G_{\sigma}/P)$ for a parabolic subgroup $P \leq G_{\sigma}$. Then the SU-rank of the associated (untwisted) group $G$ satisfies $SU(G) \leq 8n^{2} + 2n$. 
    
    In particular, $SU(G_{\sigma}) \leq 8n^{2} + 2n$, and, if $G$ is of type $A_2$, then $n > 1$.
\end{prop}

\begin{proof}
    Since any parabolic subgroup is conjugate to a standard one, we reduce to the case that $P = P_{J}$ as above. Defining $I$ as in Lemma \ref{lem: lifting}, we have that $\mathrm{SU}(G/P^{*}_{I}) \leq 2n$. We know that $\mathrm{dim}(G/P^{*}_{I})$ is greater than or equal to the Lie rank of the Chevalley group $G=G(K)$ and that $SU(G/P^{*}_{I}) = \mathrm{dim}(G/P^{*}_{I}) \cdot SU(K)$. Therefore, the Lie rank of $G$ is at most $\frac{2n}{SU(K)}$, which entails, by Fact \ref{fact: dim bound 2}, that 
    $$
    \mathrm{dim}(G) \leq 2 \left(\frac{2n}{SU(K)}\right)^{2} + \left( \frac{2n}{SU(K)}\right),
    $$
    and so 
    $$
    SU(G) = SU(K) \cdot \mathrm{dim}(G) \leq 8n^{2} + 2n.
    $$
    Since $G_{\sigma} \leq G$ we have that $SU(G_{\sigma}) \leq SU(G) \leq 8n^{2} + 2n$, as desired.
    
    Finally note that if $G$ is of type $A_2$ then the Lie rank of $G$ is $2$ and is at most $\frac{2n}{SU(K)}$. If we had $n=1$ then $2 \leq \frac{2}{SU(K)}$ so $1=SU(K)=2SU(F)$, which cannot happen as $SU(F) \geq 1$.
\end{proof}

To conclude, this effectively reduces the twisted Chevalley case to the untwisted case, at the cost of multiplying by $4$.  The argument only handles the case of standard actions, but for the non-standard actions, we have the Burness-Liebeck-Shalev bound (Theorem~\ref{fact: small base}). So we have the following result.

\begin{prop} \label{prop: almost simple case}Suppose $(G,X)$ is a definably primitive pseudo-finite permutation group of finite SU-rank of almost simple type. Then if $SU(X) = n$, we have $SU(G) \leq 8n^{2} +2n$.  
\end{prop}

\begin{proof}Recall first that, when we say that $(G,X)$ is an almost simple group we in particular have that $(G,X)$ is primitive. By Lemma \ref{lem: nonstandard bound}, if $(G,X)$ is of non-standard type, then we have $SU(G) \leq 6n$.  Therefore, we may assume that the action is standard, so the simple group $\mathrm{Soc}_{d}(G)$ is of classical type.  Let $S = \mathrm{Soc}_{d}(G)$. As $(G,X)$ is primitive and $S \unlhd G$ is a normal subgroup, we know that the action of $S$ on $X$ is transitive. The action being standard entails that for $x \in X$, the stabiliser $S_{x}$ is the set-wise stabiliser of some subspace of the natural module for the classical pseudo-finite simple group $S$, which by Theorem \ref{th:wilson}, is either a Chevalley or twisted Chevalley group over a pseudo-finite field.  Parabolic subgroups are, by definition, the setwise stabilisers of subspaces of the natural module, so we identify $X$ with $S/P$ for some parabolic subgroup $P \leq S$.  In the case that $S$ is a Chevalley group, we obtain $SU(S) \leq 2n^{2} + n$ by Proposition \ref{prop: Chevalley bound}.  Otherwise, $S$ is a twisted Chevalley group so we get $SU(S) \leq 8n^{2} + 2n$ from the last proposition. Finally, because $G$ is of finite $SU$-rank, we know that $[G : S] < \infty$ (Lemma~\ref{lem: red to simple}) so we get 
    $$
    SU(G) = SU(S) \leq 8n^{2} + 2n,
    $$
    as desired.
\end{proof}



\subsection{Simple diagonal actions}

Before getting into the details of our bound in this case, we make some remarks on the shape of primitive permutation groups of simple diagonal action type. If $(H,Y)$ is a finite primitive permutation group of simple diagonal action type with $\mathrm{Soc}(H) = T^{k}$ for $T$ a non-abelian finite simple group, then, identifying $T^{k} \cong \mathrm{Inn}(T)^{k}$, $H$ may be viewed as a subgroup of 
    $$
    W = \{(a_{1}, \ldots, a_{k})\pi \in \mathrm{Aut}(T)^{k} \rtimes \mathrm{Sym}_{k} : a_{i} \equiv a_{j} \mod \mathrm{Inn}(T) \text{ for all }i,j\}.
    $$
    Letting $T_{i}$ denote the image of $\mathrm{Soc}(H)$ under the $i$th coordinate projection, we know $H$ acts by conjugation on the set $\{T_{1}, \ldots, T_{k}\}$ and thus we have $N = \bigcap_{i = 1}^{k} N_{H}(T_{i})$ is a  normal subgroup of $H$ and clearly the quotient $H/N$ is isomorphic to a subgroup of $\mathrm{Sym}_{k}$. Moreover, the natural map $N/T^{k} \to N_{H}(T_{1})/T_{1}C_{H}(T_{1})$ is injective, by the definition of $W$, since if a tuple in $W$ induces an inner automorphism of $T_{1}$ then it must induce an inner automorphism on all $T_{i}$'s and hence be an element of $T^{k}$.  In the following lemma, we use this, together with Fact \ref{fact:field-autom}, to argue that if $(G,X)$ is a pseudo-finite primitive permutation group of finite $SU$-rank of simple diagonal action type, then the $SU$-rank of $G$ and the $SU$-rank of $X$ can both be calculated in terms of the $SU$-rank of $T$, where $T$ is a simple factor of the socle of $G$. 


\begin{lem}\label{lemma:diagonal}
Suppose $(G,X)$ is a definably primitive pseudo-finite permutation group of finite $SU$-rank.  If $(G,X)$ is of simple diagonal action type, then 
$
SU(G) \leq 2SU(X).
$
\end{lem}

\begin{proof}
     Let $\mathrm{Soc}_{d}(G) = T^{k}$ for a non-abelian simple pseudo-finite group $T$.  We know by Lemma~\ref{lem:socle} that each factor $T_{i}$ is definable and thus so is the group $N = \bigcap_{i = 1}^{k} N_{G}(T_{i})$.  By \L o\'s' Theorem and the paragraph preceding the lemma, we know that $[G: N] \leq k!$ so we have $SU(G) = SU(N)$. Similarly, we know that the natural map $$N/\mathrm{Soc}_{d}(G) \to N_{G}(T_{1})/T_{1}C_{G}(T_{1})$$ is injective.  By Fact~\ref{fact:field-autom}, we know that there can be no infinite definable group of outer automorphisms of $T_{1}$ and hence $N_{G}(T_{1})/T_{1}C_{G}(T_{1})$ is finite.  This entails $SU(G) = SU(N) = SU(\mathrm{Soc}_{d}(G))= SU(T^{k}) = k SU(T)$.  

    Fix $x \in X$.  We may definably identify $X$ and $N/N_{x}$.  Moreover, we know that, in primitive permutation groups of simple diagonal action type, the stabiliser $N_{x}$ is a diagonally embedded copy of $T$ so any/all of the coordinate projections from $N_{x}$ to $T_{i}$ are isomorphisms.  This entails that 
    $$
    SU(X) = SU(N/N_{x}) = SU(N) - SU(T) = (k-1)SU(T).
    $$

    Putting it all together, we have 
    $$
    SU(G) = k SU(T) = \frac{k}{k-1}SU(X).
    $$
    As $\frac{k}{k-1} \leq 2$ for all $k\geq 2$, we obtain the desired inequality.
\end{proof}

\subsection{Conclusion} 

Before concluding our main theorem, we still need the following easy lemma.

\begin{lem}\label{lem:productaction}
    Suppose there is a function $\rho : \mathbb{N} \to \mathbb{N}$ such that, whenever $(G,X)$ is a definably primitive permutation group of finite $SU$-rank of almost simple or simple diagonal action type, $SU(G) \leq \rho (SU(X))$. Assume that $\rho$ has the property that $\rho(ab) \geq a \rho(b)$ for all $a,b \in \mathbb{N}$. Then whenever $(G,X)$ is a definably primitive permutation group of finite $SU$-rank of product action type, we have $SU(G) \leq \rho (SU(X))$. 
\end{lem}

\begin{proof}
Since $(G,X)$ is a product action, we know there is $(H,Y)$ of either almost simple or simple diagonal action type and $\ell$ such that $G \leq H \mathrm{wr} \mathrm{Sym}_{\ell}$ and $X = Y^{\ell}$.  Then we calculate 
$$SU(G) \leq SU(H \mathrm{wr} \mathrm{Sym}_{\ell}) = \ell \cdot SU(H) \leq \ell \cdot \rho(SU(Y)) \leq \rho (SU(X)),$$
as desired. 
\end{proof}

Now we come to a proof of Theorem~\ref{th: main}:

\begin{thm} \label{th: main}
There is a function $\rho : \mathbb{N} \rightarrow \mathbb{N}$ so that if $(G,X)$ is a pseudo-finite definably primitive permutation group of finite $SU$-rank then $SU(G) \leq \rho( SU(X))$.

More precisely, if some/each point stabiliser $G_{x}$ is infinite, then, setting $r=SU(X)$, the following holds. \begin{enumerate}
\item If ${\rm Rad}(G)\neq 1$ then $SU(G) \leqslant r + (r^2+1)r$.
\item If ${\rm Rad}(G)= 1$ then one of the following holds. \begin{enumerate}
\item $(G,X)$ is an almost simple group and $SU(G)  \leq 8r^{2} + 2r$.
\item $(G,X)$ is of simple diagonal action type and $SU(G) \leqslant 2r$.
\item $(G,X)$ is of product action type and $SU(G)  \leq 8r^{2}+2r$.
\end{enumerate} 
\end{enumerate}
\end{thm}

\begin{proof}
If $(G,X)$ is a definably primitive permutation group of finite $SU$-rank, then we may assume that each point stabiliser $G_{x}$ is infinite, else $SU(G) = SU(X)$, as explained at the beginning of the section. Thus, by Theorem \ref{thm:LMT}, we may assume $(G,X)$ is primitive. Then if $\mathrm{Rad}(G) \neq 1$, then Theorem \ref{thm:LMT} implies that $(G,X)$ is of affine type and case (1) is handled by Lemma \ref{lem:affinecase}. If $\mathrm{Rad}(G)  = 1$, then, by Theorem \ref{thm:LMT}, $(G,X)$ is of almost simple, simple diagonal action, or product action type. The case of almost simple type (2)(a) is handled by Proposition \ref{prop: almost simple case} and the case of simple diagonal action type 2(b) is handled by Lemma \ref{lemma:diagonal}.  Since the bound $\rho(r) = 8r^{2} + 2r$ works for both the almost simple and simple diagonal action types and satisfies $\rho(ab) \geq a \rho(b)$ for all $a,b \in \mathbb{N}$, we obtain the bound for product action type 2(c) from Lemma \ref{lem:productaction}.
\end{proof}

\begin{rem}
In the end, our use of the hypothesis `finite $SU$-rank' is quite light, and can perhaps be relaxed to `finite-dimensional with fine and additive dimension'. We use the fact that we are in the narrower context of finite $SU$-rank in exactly two places: we use indecomposability in the proofs of Theorem~\ref{th:primitive} and Lemma~\ref{lem:socle}, and we use Fact~\ref{fact:field-autom}. As is explained before Question~\ref{q:prim}, it should not be hard to generalise Lemma~\ref{lem:socle} to the context of finite-dimensional groups with fine and additive dimension. It also seems likely that a pseudo-finite field $F$ has no infinite set $B$ of field automorphisms so that the pair $(F,B)$ has a finite-dimensional theory with fine and additive dimension. Therefore we believe that a positive answer to Question~\ref{q:prim}(b) would generalise Theorem~\ref{th: main} to the context of finite-dimensional groups with fine and additive dimension.
\end{rem}

\section{The case $SU(X)=1$}\label{sec:trank1}The following result is a combination of \cite[Theorem 1.3]{Elwesetal} and \cite[Theorem 1.16]{Zou2020}.

\begin{thm}[Elwes-Macpherson-Ryten-Jaligot and Zou]\label{th:rank1} Let $(G,X)$ be a pseudo-finite definably primitive supersimple group of finite $SU$-rank. Assume that $SU(X)=1$. Then exactly one of the following holds.\begin{enumerate}
\item $SU(G)=1$ and ${\rm Soc}_d(G)$ is either a divisible torsion-free abelian group or an elementary abelian $p$-group, has finite index
in $G$, and acts regularly on $X$.
\item $SU(G) = 2$, ${\rm Soc}_d(G)$ is abelian and so regular and identified with $X$. There is an interpretable pseudo-finite field $F$ of $SU$-rank $1$, such that $(G,X)$ is definably isomorphic to $(F^+\rtimes H, F^+)$ for some finite index subgroup $H \leqslant F^\times$ .
\item $SU(G) = 3$ and there is an interpretable pseudo-finite field $F$ of $SU$-rank $1$ so that ${\rm Soc}_d(G)\cong {\rm PSL}_2(F)$ is of finite index in $G$, ${\rm PSL}_2(F)\leqslant G \leqslant  {\rm P\Gamma L}_2(F)$, and $X$ can be identified with ${\rm PG}_1(F)$ in such a way that the action of $G$ on ${\rm PG}_1(F)$ is the natural one. 
\end{enumerate}
\end{thm}

In \cite[Theorem 1.3]{Elwesetal}, the authors proved Theorem~\ref{th:rank1} under the assumption `$Th(G)^{eq}$ eliminates $\exists^\infty$', and this assumption was used in many places of the proof. In \cite{Zou2020} Zou generalised \cite[Theorem 1.3]{Elwesetal} to the context of finite-dimensional groups with fine and additive dimension, without the assumption on elimination of $\exists^\infty$ (But with some extra assumptions in (2) and (3). Also her conclusion in (2) is only up to finite index.). Together these results prove Theorem~\ref{th:rank1}. In what follows we give an alternative proof for it, which we think should be seen as a nice application of \cite{LMT} (indeed, \cite{LMT} is also used in \cite{Elwesetal}).

\begin{proof}Let $(G,X)$ be as in Theorem~\ref{th:rank1}. If $SU(G)=SU(X)=1$ then (1) follows directly from Proposition~\ref{propo:socle}. So we may assume that $SU(G) > SU(X)$; hence the point stabilisers are infinite and $(G,X)$ is of one of the types: affine, almost simple, simple diagonal action or product action (as explained in the beginning of Section~\ref{sec:bounds}).

Suppose $SU(G)=2$. Then $G$ is soluble-by-finite \cite[Corollary 5.2]{Wagner2020} and thus an affine group. So Lemma~\ref{lem:affinecase} immediately gives us (2).

Let then $SU(G) \geqslant 3$. Then the group $(G,X)$ cannot be of affine, simple diagonal action or product action type: since $SU(X)=1$ the affine case cannot happen by Lemma~\ref{lem:affinecase}. If $(G,X)$ was of simple diagonal action type then, by Lemma~\ref{lemma:diagonal}, $SU(G) \leqslant 2$, which is not the case. The group $(G,X)$ cannot be of product action type as for otherwise there would be an associated definable permutation group $(H,Y)$ with $X=Y^\ell$ (for $\ell \geqslant 2$) and $SU(Y) \geqslant 1$, so we would get $SU(X) \geqslant 2$ which is not the case.

So $(G,X)$ is an almost simple group. Setting $\sim$ a relation on $X$ defined by $x \sim y$ if and only if $|G_x: G_x \cap G_y| < \infty$, Lemma~\ref{lem:congruencefact} ensures that $\sim$ has exactly one class or all its classes are trivial. In the latter case $SU(X)=1$ implies that there are at most finitely many $G_x$-orbits. In the former case $\widetilde{N}_G(G_x)=G$ for all $x\in X$ and one gets a contradiction as in the proof of Theorem~\ref{th:primitive}. Set $S={\rm Soc}_d(G)$. Then $S$ is a (twisted) Chevalley group over a pseudo-finite field $F$, and, as there are at most finitely many $S_x$-orbits, it follows from \cite[Theorem 2]{Seitz74} that the action of $S$ on $X$ is on the cosets of a parabolic subgroup $P$ of $S$. Recall also that we have $SU(G)=SU(S)$  (Lemma~\ref{lem: red to simple}). 

So it is enough to show that $S\cong {\rm PSL}_2(F)$: if this is the case, then, by Lemma~\ref{lem: dimension inequality}, $SU(S)={\rm dim}( {\rm PSL}_2(F))SU(F)=3 SU(F)$. At the same time a Borel subgroup $B\cong F^+ \rtimes F^\times$ of $S$, which in ${\rm PSL}_2(F)$ is also a maximal parabolic subgroup, has $SU$-rank $2 SU(F)$ and $SU(X)=SU(S/B)=1$. So $$3SU(F)=SU(S)=SU(B)+SU(S/B)=2 SU(F)+1$$ and hence $SU(F)=1$. We get that $SU(G)=3$ and (3) holds. Below we observe that indeed $S\cong {\rm PSL}_2(F)$.

The assumption $SU(X)=1$ together with Lemma~\ref{lem: nonstandard bound} and Propositions~\ref{prop: Chevalley bound} and~\ref{prop: bound twisted} (recall also Proposition~\ref{prop: almost simple case}) guarantee that one of the following holds. \begin{enumerate}[(a)]
\item $(G,X)$ is non-standard and $SU(G) \leqslant 6$.
\item $(G,X)$ is standard, $S$ is a Chevalley group of classical type and $SU(S) \leqslant 3$.
\item $(G,X)$ is standard, $S$ is a classical twisted Chevalley group obtained from a Chevalley group $H$ such that $H$ is not of type $A_2$ and $H$ is of $SU$-rank at most $10$. 
\end{enumerate}

Let $P$ be a parabolic subgroup of $S$. If $S$ is a Chevalley group then, by Fact~\ref{fact: sam fact}, $1=SU(X)=SU(S/P) \geqslant n$, where $n$ is the Lie rank of $S$. So the Lie rank of $S$ is $1$ and hence $S \cong {\rm PSL}_2(F)$. So the case (b) gives us (3), and from now on we may assume that $S$ is a twisted Chevalley group. Let $H$ be the associated Chevalley group. If (c) happens then, since ${\rm PSL}_2$ (i.e. type $A_1$) has no twisted  analogue, $H$ is of type $A_n$ ($n \geqslant 3$) or $D_n$ ($n \geqslant 4$) and hence of dimension $n(n+1)+n$ or $2n(n-1)+ n$, respectively (see the proof of Lemma~\ref{lem: dimension inequality}). So $SU(H)={\rm dim}(H)SU(F)> 10$, which is not the case. 

Finally assume that (a) happens. We have already observed that the action is parabolic and that $S$ may be assumed to be a twisted group. Proposition~\ref{prop: bound twisted} (and the argument above) thus shows that $S$ is not of classical type $^2A_n$ ($n \geqslant2$) or $^2D_n$ ($n \geqslant 4$). So it remains to show that (a) cannot happen when $S$ is of exceptional type $^2B_2$, $^2G_2$, $^2F_4$, $^2E_6$ or $^3D_4$.

By \cite[Chapter 5]{Ryten2007}, $S$ is uniformly bi-interpretable with a pseudo-finite field $F$ or, in the case of Suzuki and Ree groups, with a a pseudo-finite difference field $(F, \sigma)$. We now observe that even when the $SU$-rank of $F$ (resp. $(F, \sigma)$) is as small as possible, $SU(F)=1$ (resp. $SU((F, \sigma))=1$), we violate the bound in (a). As is explained in the proof of \cite[Lemma 5.15]{Elwesetal}, on an ultraproduct of finite fields, by the results in \cite{Chatzidakis1992} the $SU$-rank of a definable set $Y$ is determined by the approximate cardinalities of the corresponding $Y_i$. The same holds for the difference fields over which the Suzuki and Ree groups are defined by \cite[Theorem 5.8]{Elwes2005} and \cite[Corollary 5.4]{EM08}. By bi-interpretability this is transferred to $S\equiv \prod_{i\in I}S_i/\mathcal{U}$. That is, if $S_i$ is bi-interpretable
with $\mathbb{F}_{q_i}$, then an ultraproduct of uniformly definable sets (each of cardinality roughly $\mu q^d$ ) has $SU$-rank $d$. The finite groups $^2B_2(q)$, $^2G_2(q)$ and $^2F_4(q)$ are defined over finite fields with $q= 2^{2k+1}$, $q=3^{2k+1}$ and $q= 2^{2k+1}$, respectively. We also consider $^2E_6(q^2)$ and $^3D_4(q^3)$. It follows by consideration of orders (see e.g. \cite[Theorem 14.3.2]{Carter1971}) that either $SU(S) > 6$ or $S$ is of type $^2B_2$ and $SU(S)=5$. In the former case we violate the bound in (a). If $S$ is of type $^2B_2$ (recall that $(G,X)$ is non-standard with parabolic action) then it is known that there is a base of size $3$ \cite[Theorem 3]{Burnessetal2009}. So (by repeating the arguments in Lemma~\ref{lem: nonstandard bound}), if $S$ is of type $^2B_2$ then $SU(S)\leqslant 3$, which is not the case. We have observed that (a) cannot happen.\end{proof}

\subsection*{Acknowledgements}

We would like to thank Sam Evens for helpful conversations on algebraic groups.  We would also like to thank the referee for a thoughtful and careful reading, which greatly improved the paper.

\bibliographystyle{plain}
\bibliography{ulla2021.bib}{}

\end{document}